\documentclass[amsart,11pt]{article}

\usepackage{amsfonts}
\usepackage{amsmath}
\usepackage{amssymb}
\usepackage{amsthm}
\usepackage{bbm}
\usepackage{graphicx}
\usepackage{gensymb}
\usepackage{comment}
\usepackage{placeins}
\usepackage{geometry}
\usepackage{mathtools}
\usepackage{hyperref}
\usepackage{color}
\usepackage[titletoc]{appendix}
\usepackage{algorithm} 
\usepackage{algorithmic}
\newcommand{\R}{\mathbb R}
\newcommand{\F}{\mathcal F}
\newcommand{\T}{^\mathsf{T}}
\newcommand{\E}{\mathbb E}

\newcommand{\inv}{^{-1}}
\newcommand{\e}{\text{e}}
\newcommand{\trace}{\mathrm{trace}}
\newcommand{\rme}{\mathrm{e}}
\newcommand{\rmi}{\mathrm{i}}
\newcommand{\ord}{ r}
\newcommand{\rmd}{\mathrm{d}}
\newcommand{\tr}{\mathrm{trace}}

\newtheorem{theorem}{Theorem}

\newtheorem{lemma}{Lemma}
\newtheorem{prop}{Proposition}
\newtheorem{definition}{Definition}

\usepackage{tcolorbox}
\usepackage{lipsum}
\tcbuselibrary{skins,breakable}
\usetikzlibrary{shadings,shadows}

\def\toby{\color{black}}

\def\TLS{\color{black}}

\title{Effective New Methods for Automated Parameter Selection in Regularized Inverse Problems}

\author{Toby Sanders, Rodrigo B. Platte, Robert D. Skeel \\ \small{School of Mathematical and Statistical Sciences, Arizona State University, Tempe, AZ, USA.} }
\date{}

\begin{document}
\maketitle

\begin{abstract}
The choice of the parameter value for regularized inverse problems is critical to the results and remains a topic of interest. This article explores a criterion for selecting a good parameter value by maximizing the probability of the data, {{with no prior knowledge of the noise variance}}.  These concepts are developed for $\ell_2$ and consequently $\ell_1$ regularization models by way of their Bayesian interpretations.  
Based on these concepts, an iterative scheme is proposed and demonstrated to converge accurately, and analytical convergence results are provided that substantiate these empirical observations.  For some of the most common inverse problems, including MRI, SAR, denoising, and deconvolution, an extremely efficient algorithm is derived, making the iterative scheme very attractive for real case use.  The computational concerns associated with the general case for any inverse problem are also carefully addressed.  A robust set of 1D and 2D numerical simulations confirm the effectiveness of the proposed approach.
\end{abstract}

\section{Introduction}  
Image and signal denoising and reconstruction problems are important research topics due to their wide range of applications, including medical diagnosis, defense, and basic scientific research \cite{lustig2007sparse,1257394,4587391,bhattacharya2007fast,horowitz1994three,leng2010cryogenic}.  These problems arise when an object or image, which we denote by $u$, cannot easily be observed in a straightforward manner, e.g. brain imaging.  In a linear model, when $u$ must be measured in a indirect fashion, measurements of $u$ are encoded into a data vector of the form $b = Au + \epsilon$, where $A$ is a linear operator and $\epsilon$ is a noise term inherent to the sensing mechanism of the application.  Then the image reconstruction problem is to recover or \emph{decode} the most accurate representation of $u$ given $A\in \R^{m\times n}$ and $b$, and any additional prior information.  {The topic of this article is the proper choice of the regularization parameter in the reconstruction model which balances the data fitting with the priors.}

Inverse problems are typically characterized as ill-posed, resulting in solution maps that are sensitive to the noise term.  To alleviate this issue, it is common to implement regularization techniques that promote favorable solutions based on prior knowledge of the behavior of the target signal.  
For example, in the continuous formulation, an order 1 Tikhonov regularization scheme sets $R(u) = \int_{\Omega} |\nabla u(x)|^2  \, dx $, and the model minimizes a weighted sum of $R(u)$ with $\| Au - b\|_2^2$.  This formulation intuitively recovers \emph{smooth} solutions with small variation \cite{tikhonov2013numerical}.   
The regularized solutions considered in this article takes the form
\begin{equation}\label{reg-l1}
u_\lambda = \arg \min_{u\in \R^n} \| Au-b \|_2^2 + \lambda \| Tu \|_p^p,
\end{equation}
for $p=1,2$.  The main focus of this article is the choice of the important parameter $\lambda>0$ that balances the regularization with the data fitting term.

In (\ref{reg-l1}), the case $p=2$ {\TLS{yields the discretized model}} for Tikhonov regularization, and finite difference matrices are often used for $T$ to approximate derivatives \cite{ROF,TGV,SGP-ET}.  The case $p=1$ is typically referred to as the \emph{compressed sensing} formulation when $m<n$ \cite{candes2006robust,CSincoherence}.  While the Tikhonov case is less computationally expensive using conjugate gradient (CG) methods, it is well documented that in many applications the $\ell_1$ regularized solutions can be superior, hence they are also of interest in the current work.

{\TLS{Generally, the main benefit of working with Tikhonov regularization is computational.  For the most challenging problems where the best possible results are necessary, more computationally intensive techniques can yield better results by using for example \emph{learning-based} priors \cite{romano2017little,zhang2017learning,sun2019online} or $\ell_1$ regularization models as indicated in some of our results.  Nonetheless, Tikhonov regularization remains the regularizer of choice for a number of problems with sufficient data but where a fast and reliable reconstruction is needed.  For example, it was recently studied in the case of rapid blind deconvolution of low noise images \cite{xue2012sure,xue2014novel} and for real time processing of very large data sets \cite{slagel2019sampled,buccini2017regularizing,buccini2017iterated}.  Moreover, even with the most recent wave of \emph{learning-based} prior models, the alternating direction method of multipliers (ADMM) approach most often used as the computational tool to optimize these models involves solving intermediate Tikhonov regularization problems \cite{boyd2011distributed,romano2017little,venkatakrishnan2013plug}.  Hence, the work developed here could potentially be leveraged to improve these algorithms, e.g. in the way we leverage our approach for novel $\ell_1$ regularized parameter selection.
}}

In general, the parameter $\lambda$ should be larger for smaller signal-to-noise ratio (SNR) in $b$, and vice versa.  The SNR is unknown for most applications, and even when it is this does not immediately inform us what value $\lambda$ should take.  Perhaps most commonly researchers choose $\lambda$ based on ``experience."  While this approach often leads to suitably pleasing results, having a robust automated approach eliminates any user bias, provides more broadly applicable results, and saves time for the users.

This article proposes a new criterion and procedure for the parameter selection, {{which uses the equivalent maximum a posteriori (MAP) interpretation of regularized inverse problems for a Bayesian approach}} \cite{kaipio2006statistical}.  Using the MAP interpretation of (\ref{reg-l1}), we consider the parameter $\lambda$ to be optimal when providing maximum evidence (ME), i.e. it maximizes the likelihood of the data, $b$, which is related to a marginalized maximum likelihood (ML) estimation \cite{Bishop}.  {The main contributions of this article, which are expanded upon further below, may be summarized by the following:
\begin{enumerate}
\item Proof of concept of the ME criterion for noise estimation and effective parameter selection, showing comparable results to the leading methods where the noise level is assumed known.
\item Algorithms and computational methods for very efficient implementation, particularly for important applications such as MRI and deconvolution, and a novel nonlinear iteration scheme for finding the parameter that is robust and efficient

\item Extension of the recovered $\ell_2$ parameter value to obtain an accurate $\ell_1$ parameter, which is optimal as $m\rightarrow \infty$.
\end{enumerate}}

{In light of the proposed ME criterion, several theoretical results are first developed that are necessary for the algorithm.}  Then, based on these results, we propose a fixed point iterative scheme that updates two parameters, $\sigma$ and $\eta$, that together determine $\lambda$.  
The value $\sigma^2$ is the variance on the random noise vector $\epsilon \sim N(0, \sigma^2)$, which, {{contrary to most parameter selection methods}}, we assume to have no prior knowledge of.  The parameter $\eta$ is related to the \emph{regularity} or scale of the solution, and we call $\eta^2$ the variance of the signal.  This iterative scheme is developed for the Tikhonov regularized problem and is shown to converge accurately in relatively few iterations (e.g. 5 or 10). 

Revisiting the Bayesian formulation, we show how the recovered parameters $\sigma$ and $\eta$ for $\ell_2$ regularization {{give very good estimates for the $\ell_1$ regularization parameter}}, which is more appealing for many applications.  This contribution, while simple in derivation, is shown to be very effective, and we believe to be one of the critical pieces of this work.  Finally some analysis of the convergence of the proposed fixed point scheme is given, which indicates there are typically two and possibly more nontrivial parameters $\lambda$ satisfying the equations derived for the ME parameter.  However, this analysis also indicates that when the appropriate regularization operator $T$ is used, the only stable parameter (and hence the one found with the proposed algorithm) is indeed the ME parameter, which converges for a very large range of starting values.

There are several computational considerations to address for our approach.  In particular, each iteration requires the trace of a matrix which is infeasible to compute directly for large imaging problems.   However, for many important inverse problems including image denoising, deconvolution, and image reconstruction from Fourier data (e.g. SAR, MRI, and deconvolution), we analytically develop the necessary ingredients that allow us to compute the value of the necessary traces and solution exactly at the cost of essentially just one fast Fourier transform (FFT) resulting in an extremely fast algorithm.  This makes our scheme very appealing and practical for real case use.  The general case of any inverse problem is also addressed, in which the authors implemented trace estimation procedures using random vectors, which is closely related to the trace estimation procedures already used for UPRE. 

{The remainder of this article is organized as follows.  In the subsection below we discuss previous work.  In section \ref{sec: main} and \ref{sec: alg} the necessary background and theoretical developments are provided, which leads to the fixed point iteration for finding the $\ell_2$ regularization parameter.  The general outline of the full computational algorithm is also given, and a few simple examples demonstrate the approach and its potential.  This is followed in \ref{sec: UPRE} and  \ref{sec: tomo} with several additional simulated test problems and comparison with other methods.  In section \ref{sec: L1} we demonstrate how the $\ell_2$ parameters recovered using our algorithm may be used to obtain an accurate estimate for the $\ell_1$ regularization parameter at no computational cost, and simple examples show that this is very effective.  In section \ref{sec: acc} the fast versions of the algorithm are developed for problems such as MRI, SAR, and deconvolution, and they are shown to provide several orders of magnitude in speed up.  To avoid complicating the work, some of the details necessary for this algorithm are provided in the appendix.  The convergence analysis of this algorithm is provided in section \ref{sec: conv} and some of these details are also provided in the appendix.}

\subsection{Previous work}
There have been a number of automated and semi-automated approaches proposed for choosing parameter selection in inverse problems \cite{vogel2002computational}, {{which rely on various different criteria that characterize a \emph{good} parameter}}.  These include the L-curve method \cite{Hansen-l-curve,Calvetti2000423} and generalized cross validation (GCV) \cite{golub1979generalized}.  {{The L-curve method is mainly an empirical method, while GCV provides a reasonable criterion based on having a model where solutions from subsets of data fit complimentary data sets.  There is also the discrepancy principle, which enforces the condition that the regularized solution matches the assumed known noise variance so that $\| Au -b \|_2^2 \approx m \sigma^2$.  This method is not generally preferred due to unfavorable empirical evidence \cite{vogel2002computational,giryes2011projected}, while also requiring accurate knowledge of $\sigma^2$.}} Moreover, these methods typically require one to compute many solutions over a potentially large range of possible parameters in a brute force approach to find the optimal value.  

{{The leading criteria for parameter selection are the unbiased predictive risk estimator (UPRE) \cite{mallows1973some,slagel2019sampled} and Stein's unbiased risk estimator (SURE)\cite{stein1981estimation}, which provide unbiased estimates of a squared error.  The goal then, is to minimize these estimators as functions of $\lambda$.  Denoting the true solution by $v$, UPRE gives an estimate of the predictive error by
\begin{equation}
\E  \|A(u_\lambda - v) \|_2^2  = -m\sigma^2 + \| Au_\lambda - b\|_2^2 + 2\sigma^2 \trace(A B_\lambda),\end{equation}
where $B_\lambda$ is the linear solution map, i.e. $u_\lambda = B_\lambda b$.}} The UPRE method {{has usually been applied}} to Tikhonov regularization \cite{renaut2010regularization}, hence $B_\lambda = (A\T A + \lambda T\T T)\inv A\T$, although it is applicable more generally to problems where the solution depends linearly on the data.  {{The method also assumes i.i.d. Gaussian noise and prior knowledge of the variance $\sigma^2$}}, and the minimization of UPRE typically requires an exhaustive search over $\lambda$ \cite{vogel2002computational}, similar to procedures for the L-curve method.

{{The SURE estimator is similar but more general than UPRE due to its application to nonlinear inverse models.  In fact, the UPRE method just described can be easily derived from SURE, although the pure form of SURE for inverse problems was originally considered for denoising and threshold selection \cite{donoho1995adapting,zhang1998adaptive}. 
The fact that SURE can be applied to nonlinear solutions has generated interest for its application to $\ell_1$ regularization models \cite{zou2007degrees,donoho1995adapting, ramani2012regularization}. More recent developments \cite{eldar2009generalized,giryes2011projected} have generalized SURE (GSURE) to minimize the expectation of $\| P(u_\lambda -v) \|_2^2$, where $P = A\T (A\T A)^{-1} A$ is the projection operator onto the range of $A\T$.  This same work further extends the method to families of exponentials (e.g. non-i.i.d. noise models), which is outside the scope of this work.  As observed in \cite{ramani2012regularization}, UPRE and the projected GSURE are both estimators of particular instances of weighted error norms.  Likewise to UPRE,  the literature on SURE always assumes the noise variance (or covariance matrix) is known \emph{a priori}, contrary to our approach.  For additional detailed discussion and review of these topics, see \cite{ramani2012regularization} and the references therein.}}

{\TLS{More recent work has successfully proposed learning adaptive parameters from a training set \cite{dong2018learning}.  Certainly, developing a training set to learn from is far from ideal for the parameter selection problem if there is a suitable alternative.  The statistical criteria for which we develop an algorithm
are classical methods, whose conceptual basis is easy to explain.
Such methods will be used for a long time
by many researchers, and for comparison purposes by researchers
developing new methods (see e.g. recent examples in \cite{dong2018learning,slagel2019sampled}), and for these reasons we do not consider this approach in this work.}}

\section{Bayesian Formulation and Algorithm for Finding $\lambda$}\label{sec: main}
We begin by writing an equivalent expression for $u_\lambda$ in (\ref{reg-l1}) as
\begin{equation}\label{MAP-main}
\begin{split}
u_\lambda & = \arg\max_u p(u|b) = \arg\max_u p(b|u) p(u) \\
& = \arg \max_u  \text{exp}\left(-\frac{\| Au - b\|_2^2 }{2\sigma^2} \right) \exp\left( - \frac{\lambda \| Tu\|_p^p}{2\sigma^2} \right) .
\end{split}
\end{equation}
We observe this expression as a maximum a posteriori (MAP) formulation from the Bayesian perspective \cite{kaipio2006statistical}, and we write $u_\lambda$ maximizes $ p(u |  b) \propto p(b|u) p(u)$. 
For a noise vector $\epsilon \sim N(0,\sigma^2 I)$, it is clear that 
\begin{equation}\label{likelihood}
 p(b|u,\sigma) = (2\pi \sigma)^{-m/2} \text{exp} \left(-\frac{\| A u - b \|_2^2}{2\sigma^2} \right) .
\end{equation}

For the prior we consider for now $\ell_2$ regularizations with parameter $\eta$, which is related to the \emph{regularity} of the signal, and we call $\eta^2$ the variance of the signal (under the map $T$).  Then our $\ell_2$ Gaussian prior in the case $T$ is nonsingular takes the form
\begin{equation}\label{prior}
 p(u| \eta) = \frac{\det T }{(2\pi \eta^2)^{n/2}} \text{exp} \left (-\frac{\| Tu \|_2^2}{2\eta^2} \right).
\end{equation}
From (\ref{likelihood}) and (\ref{prior}) we see that for a given $\sigma$ and $\eta$, which are generally unknown, we have the MAP formulation as
\begin{equation}\label{reg-bayes}
u_{\sigma, \eta} = \arg \max_u \, \text{exp} \left(-\frac{\| A u - b \|_2^2}{2\sigma^2} \right)
 \text{exp} \left (-\frac{\| Tu \|_2^2}{2\eta^2} \right) ,
\end{equation}
which one may observe is equivalent to the minimization in (\ref{reg-l1}) with $p=2$ and $\lambda = \sigma^2 / \eta^2$.  Numerically, we use the formulation (\ref{reg-l1}) to find $u_\lambda$, however we make use of the equivalent Bayesian formulation (\ref{reg-bayes}) for the analysis in finding good parameters.

Normally the $\ell_2$ prior would take the form
$p(u)\propto\exp(-\|T u\|_2^2/(2\eta^2))$.
However, this is improper in the typical case where $T$ is singular,
since it cannot be normalized.
To fix this (see section 3.4 of \cite{kaipio2006statistical}),
one may use a (nearly) flat prior on the null space of $T$,
in particular, a Gaussian with variance $\alpha$ approaching $+\infty$.
Proceed by writing the singular value decomposition of $T$ as
\[T = U\left[\begin{array}{cc}\Sigma_1 & 0\end{array}\right]
 \left[\begin{array}{c}V_1\T \\ V_2\T\end{array}\right]\]
and using
\[p(u|\eta) = C_2\exp
\left(-\frac{\|T u\|_2^2}{2\eta^2} -\frac{\|V_2\T u\|_2^2}{2\alpha^2}\right),\]
  for which
\[C_2 = \frac{\det\Sigma_1}{(2\pi\eta^2)^{(n-\ord)/2}(2\pi\alpha^2)^{ r/2}}\]
where $n-r$ is the rank of $T$.

\subsection{Iterative Algorithm for Finding the Noise and Signal Variances} \label{sec: alg}
Our goal is to find a good estimate for $\sigma$ and $\eta$ in an efficient manner. We will make use of the Bayesian interpretation of Tikhonov regularization with probability distributions in (\ref{likelihood}) and (\ref{prior}).  The parameters will be considered \emph{good} in a marginalized maximum likelihood sense for maximum evidence (ME), i.e. they maximize $p(b|\sigma, \eta)$.  We first provide the theoretical developments necessary for an ME algorithm.   

\begin{definition}
For a given $\sigma$ and $\eta$, we define $u_{\sigma,\eta}$ to be the MAP solution in (\ref{reg-bayes}), which may be equivalently expressed by the Tikhonov regularized minimizer as
\begin{equation}\label{eq: tik}
 u_{\sigma,\eta} = u_\lambda = H^{-1} A\T b = \arg \min_u \| Au - b\|_2^2 + \lambda \| Tu \|_2^2,
\end{equation}
where $H = A\T A + \lambda T\T T$ and $\lambda = \sigma^2/\eta^2$.
\end{definition}

\begin{lemma}\label{main-lem}
Consider the expectation of an arbitrary function $f(U)$, where $U$ is a random variable with the Gaussian density function in (\ref{prior}).  Then for a given $b$ with conditional expectation given by (\ref{likelihood}), the conditional expectation of $f(U)$ is
\begin{equation}
\E[f(U) \, | \, b] = \E [ f(u_\lambda + \sigma H^{-1/2} X) ],
\end{equation}
where $X \sim N(0 , I)$.
\end{lemma}
\begin{proof}
We will use capital $C$'s to denote constants independent of $u$ that can be absorbed on the outside.  We begin by writing the conditional expectation using Bayes' theorem as
\begin{align*}
\E[f(U) \, | \, b] 
&= \int_{\R^n} f(u) p(b|u,\sigma) p(u|\eta) / p(b) \, \rmd u\\
& = 
 C_1\int f(u)
\exp\left(-\frac1{2\sigma^2}(u\T H_\alpha u - 2 u\T A\T b)\right)\, \rmd u,
\end{align*}
where $H_\alpha = H + (\sigma^2/\alpha^2)V_2 V_2\T$.
Completing the square of the matrix equation in the exponential leads to
\begin{align*}
 \E[f(U)\, | \, b] 
 &= C_2 \int_{\R^n} f(u) \text{exp}\left(\frac{-1}{2\sigma^2} (u-u_{\lambda})\T H_\alpha (u-u_{\lambda}) \right) \, \rmd u\\
 & = C_3 \int_{\R^n} f(u_{\lambda}+\sigma H_\alpha^{-1/2} x) \text{exp}(-x\T x/2) \, \rmd x.
\end{align*}
Because $H$ is symmetric positive definite,
$H_\alpha^{-1/2} = H^{-1/2} +\mathcal{O}(\alpha^{-2})$, and
the limit $\alpha\rightarrow+\infty$ is well defined.
Setting $f(u) = 1$, we observe that $C_3 = (2\pi )^{-n/2}$, which completes the proof.
\end{proof}

\begin{lemma}\label{lem-2}
 Let $b$ be given with conditional density defined by (\ref{likelihood}), and let $U$ be a random variable having a density defined by (\ref{prior}), with $\sigma$ and $\eta$ considered unknown parameters.  Then the values of $\sigma$ and $\eta$ that maximize $p(b)$ satisfy the following conditional expectations:
 \begin{align}
  \sigma^2 &= \frac{1}{m} \E \left[ \|AU - b \|_2^2 \, | \, b\right] \label{eq: E1}  \\
  \eta^2 & = \frac{1}{n} \E \left[ \| TU\|_2^2 \, | \, b\right] \label{eq: E2} .
 \end{align}
\end{lemma}
\begin{proof}
Using the law of total probability leads to
 \begin{align*}
  p(b) 
  & = \int_{\R^n} p(b|u) p(u) \, du \\
  & = (2\pi \sigma^2)^{-m/2} (2\pi \eta^2)^{-n/2} \det T \int_{\R^n} \text{exp} \left(
  -\frac{\|Au-b\|_2^2}{2\sigma^2} - \frac{\| Tu \|_2^2}{ 2 \eta^2}
  \right) \, \rmd u
 \end{align*}
Differentiating this expression with respect to $\sigma$ leads to
\begin{equation}
 \frac{d}{d\sigma} p(b) 
 = - \frac{m}{\sigma} p(b)+ \sigma^{-3} p(b) \E \left[ \| AU - b\|_2^2 \, | \, b\right] .
\end{equation}
Setting this expression to zero completes the proof for (\ref{eq: E1}), and the details for (\ref{eq: E2}) are similar.
\end{proof}

\begin{theorem}\label{main-thm}
Let $U$ be a random variable with density given by (\ref{prior}), and let $b$ be given with conditional density given by (\ref{likelihood}).  Then $\sigma$ and $\eta$ which maximize $p(b)$ in terms of expectations satisfy the following equalities:
\begin{align*}
\sigma^2 & =  \| A u_\lambda - b \|_2^2  /(m-\trace(H^{-1} A^T A)) \\
\eta^2 & =  \|Tu_\lambda  \|_2^2 /(n- \lambda \trace (H^{-1}T\T T)) ,
\end{align*} 
where $H = A\T A + \lambda T\T T$ and $\lambda = \sigma^2/\eta^2$
\end{theorem}
\begin{proof}
Letting $f(u) = \| Au-b\|_2^2$ and applying Lemma \ref{main-lem} leads to
$$
\E[\| AU - b \|_2^2 \, | \, b ]= \E[\| A(u_{\lambda}+\sigma H^{-1/2} X) - b\|_2^2].
$$
Expanding this expression out and using the properties of $X\sim N(0,I)$ (see for example, Lemma 7.2 in \cite{vogel2002computational}) leads us to
\begin{equation}\label{eq: exp1}
 \E[\| AU - b \|_2^2 \, | \, b ] = \| A u_{\lambda} - b\|_2^2 + \sigma^2 \text{trace}(H^{-1} A\T A) .
\end{equation}
In a similar fashion
\begin{equation}\label{eq: exp2}
 \E[\| TU \|_2^2 \, | \, b] = \| T u_{\lambda} \|_2^2 + \lambda \eta^2 \text{trace}(H^{-1} T\T T) .
\end{equation}
Combining equations (\ref{eq: exp1})-(\ref{eq: exp2}) with Lemma \ref{lem-2} completes the proof.
\end{proof}


 Theorem \ref{main-thm} is the basis of an iteration to find $\sigma$ and $\eta$ for the ME algorithm, which is given by
\begin{align}
 \sigma_{k+1}^2 &= \| A u_{k}^\ast - b \|_2^2/(m-\text{trace}(H_k^{-1} A\T A) ) \label{eq: iter1}\\
 \eta_{k+1}^2 & = \| T u_{k}^\ast \|_2^2 / (n - \lambda_k \text{trace}(H_k^{-1} T\T T)), \label{eq: iter2}
\end{align}
 where it is implied in this case that $H_k = A\T A + \lambda_k T\T T$, $\lambda_k = \sigma_k^2/\eta_k^2$, and $u_k^\ast$ is the Tikhonov regularized solution for parameter $\lambda_k$.
This iteration is set to converge whenever
\begin{equation}\label{eq: convergence}
 \frac{\| u_{k+1}^\ast -  u_k^\ast \|_2}{\|  u_k^\ast \|_2} < tol,
\end{equation}
or until we reach some maximum number of iterations $K$, which we demonstrate only mildly depends on $\lambda_0$.

One must consider the cost of such an iteration.  Later in section \ref{sec: acc} we give the most useful results to deal with the computational cost.  The results in that section lead to an extremely efficient scheme for this iteration for some of the most common inverse problems, including denoising, deconvolution, and MRI.  For completeness however, we present the general algorithm first.  Obviously each iteration requires the minimization (\ref{eq: tik}) to find $u_k^\ast$, which for general sampling matrices $A$ is most suitably solved with a Krylov subspace method, in which case we implement a conjugate gradient method.  The biggest computational burden for the general problem is in approximating the traces in (\ref{eq: iter1}) and (\ref{eq: iter2}), for which a Monte Carlo method \cite{bai1996some} is suggested.  Specifically, the trace of any square matrix $C$ is given by $\E [X\T C X] = \trace(C)$, where $X$ is a random vector of independent Gaussians.  In its pure form this method uses a set of independent pseudo-random vectors $\{x_j\}_{j=1}^J$.  The error of the estimate is noticeably reduced if the set is first orthogonalized, notwithstanding the bias that is introduced.  One can precompute $A\T A x_j$ and $T\T T x_j$.  However, since $H$ depends on $\sigma$ and $\eta$ we must compute approximations to $H^{-1} x_j$ at each iteration. Hence at each iteration over $\sigma$ and $\eta$ we have $J+1$ solves of (\ref{eq: tik}).  An outline of the general algorithm is provided in Algorithm \ref{algo-main}.


\begin{algorithm}[!ht]
\caption{}
\label{algo-main}
\begin{algorithmic}[1]
\STATE{Inputs: $b$, $A$, $T$, $\lambda_0$.}
\STATE{Generate random vectors $\{ x_j \}_{j=1}^J$ with i.i.d entries, mean value 0 and variance 1.}
\STATE{Compute $y_j = T\T T x_j$ and $z_j = A\T Ax_j$, for $j=1,\dots , J$.}
\FOR{$k$=0 \TO K}
\STATE{Define $H_k = A\T A + \lambda_k T\T T$.}
\STATE{Numerically evaluate $u_k^\ast = H_k^{-1} A\T b$ using conjugate gradient method.}
\STATE{Numerically evaluate $w_j = H_k^{-1} x_j$ using conjugate gradient method, for $j= 1, \dots , J$.}
\STATE{Compute $\mathbb T = \tfrac{1}{J} \sum_{j=1}^J w_j\T y_j$ and $\mathbb A= \tfrac{1}{J} \sum_{j=1}^J w_j\T z_j$ to serve as inital approximations to \text{trace}$(H_k^{-1} T\T T)$ and \text{trace}$(H_k^{-1} A\T A )$, respectively.}
\STATE{Improve trace approximations by setting $\mathbb A_2 = n \mathbb A  /(\mathbb A + \lambda_k \mathbb T)$ and $\mathbb T_2 = n \mathbb T  /(\mathbb A + \lambda_k \mathbb T)$, where we have used the identity $n = \text{trace}(H_k^{-1} A\T A) + \lambda_k \text{trace}(H_k^{-1} T\T T)$.}
\STATE{Set $\sigma_{k+1}^2 = \| Au_k^\ast - b \|_2^2/(m- \mathbb A_2) $ and $\eta_{k+1}^2 = \| Tu_k^\ast \|_2^2 /(n- \lambda_k \mathbb T_2)$.}
\STATE{Set $\lambda_{k+1} = \sigma_{k+1}^2 / \eta_{k+1}^2$.}
\ENDFOR
\end{algorithmic}
\end{algorithm}


\begin{figure}[ht]
\centering
\includegraphics[width=1\textwidth,trim={0.5cm 3.5cm 2cm 3cm}, clip]{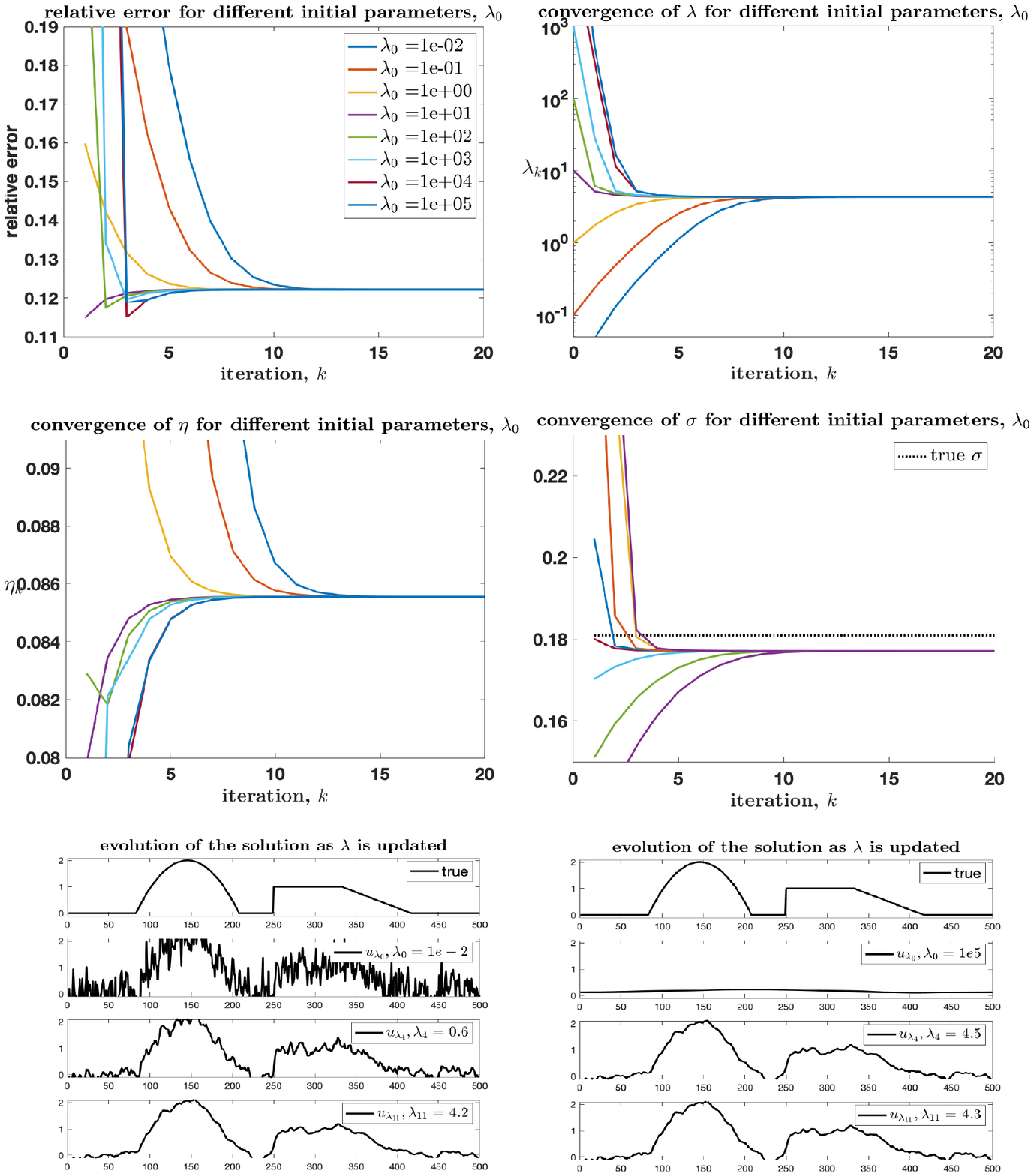}
\caption{Evolution of the parameters $\sigma$, $\eta$, and $\lambda$ for different initial parameter selections $\lambda_0$, and the evolution of two solutions.}
\label{lovely-example}
\end{figure}

We proceed with an example demonstrating the effectiveness of this approach on a 1D piecewise quadratic signal of dimension $n=500$.  The sampling matrix $A\in\R^{n\times n}$  was generated randomly with independent normally distributed entries.  For the additive i.i.d. Gaussian noise vector $\epsilon$ we chose $\sigma$ so that the signal to noise ratio (SNR) is 2.  For the regularization operator $T$, {\TLS{in order to avoid the inverse crime of prior knowledge of the quadratic signal, we used a first order finite difference to penalize the discrete first derivative.} } The only parameter left to choose is $\lambda_0$.  
We chose a series of initial $\lambda_0$'s that are equally spaced logarithmically and plot, in Figure \ref{lovely-example}, the evolution of various parameters and errors as the iteration progresses. 
At the very least, it should be evident that the convergence in this example is only mildly dependent of the initial $\lambda_0$ even at several orders of magnitude difference, and in all cases convergence is practically achieved in 10 iterations or less.  Moreover, the relative error (top left), defined by the relative difference of the reconstruction and the true solution analogous to (\ref{eq: convergence}), is observed to {\TLS{nearly monotonically decrease after each update, indicating a near optimal recovered solution for $\lambda$. }} All simulations converge the same final $\lambda$ with a maximum relative difference between the parameters to be less than $4\times 10^{-5}$.  This in turn generates very similar solutions $u_\lambda$, as evidenced by the similarity of the error plots.  The convergence of $\sigma$ and $\eta$ are also presented, and $\sigma$ is shown to accurately converge close to the true value in all cases.  Finally, for completeness, the evolution of the solutions are presented in the bottom right two panels for initial selections of $\lambda_0 = 10^{-2}$ and $\lambda_0 = 10^5$.  Many other numerical results are presented later that further confirms the findings of this example.

\subsection{Comparison with UPRE} \label{sec: UPRE}
\begin{figure}[ht]
 \centering
 \includegraphics[width=.4\textwidth]{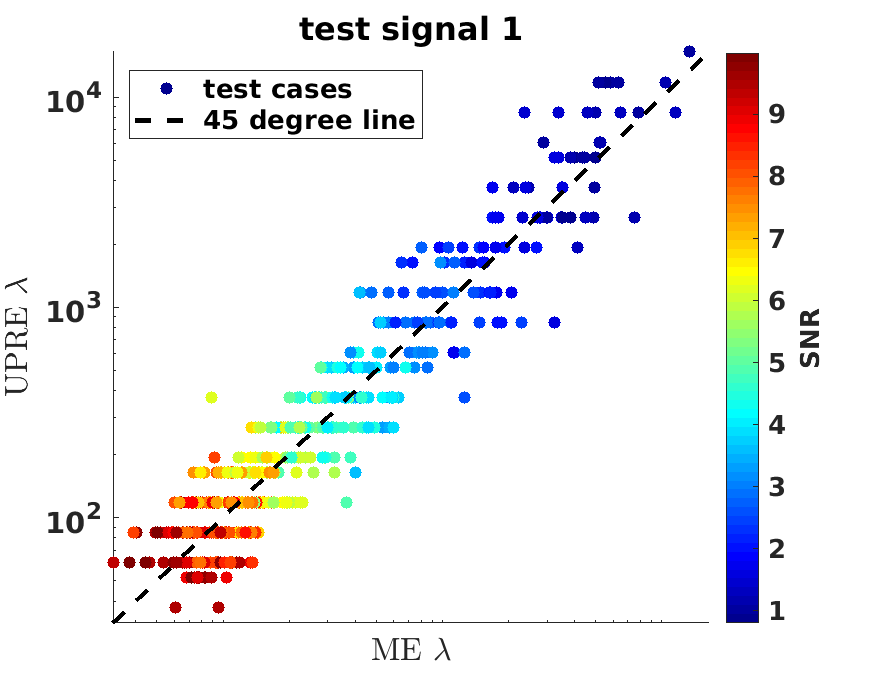}
 \includegraphics[width=.4\textwidth]{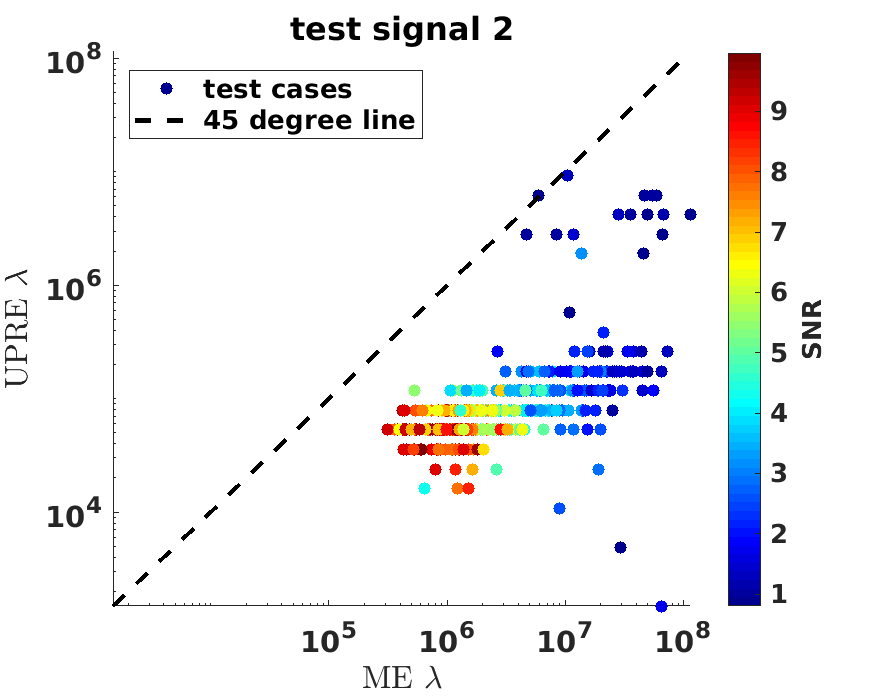} \\ 
 \includegraphics[width=.4\textwidth]{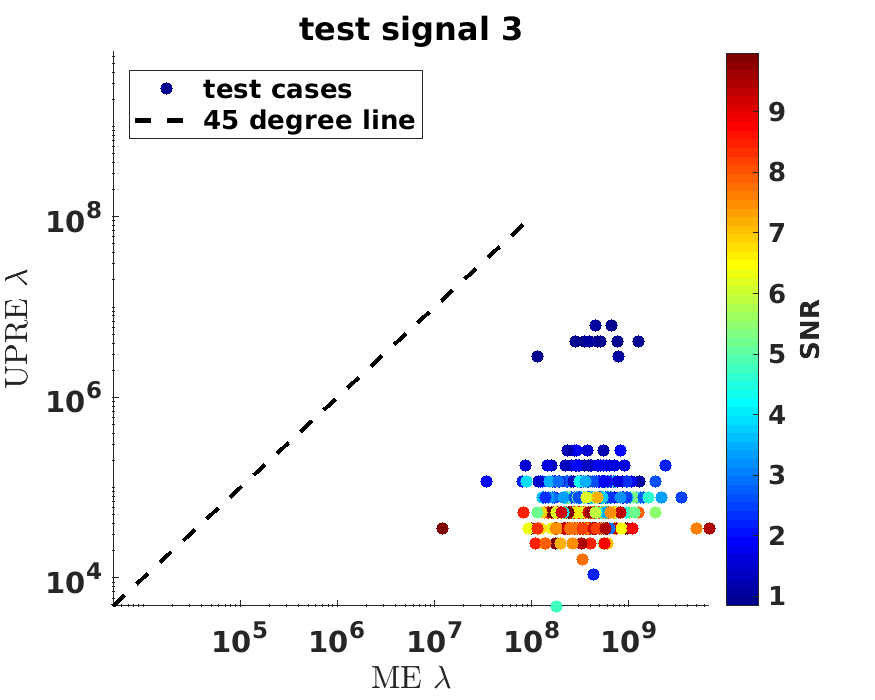}
  \includegraphics[width=.4\textwidth]{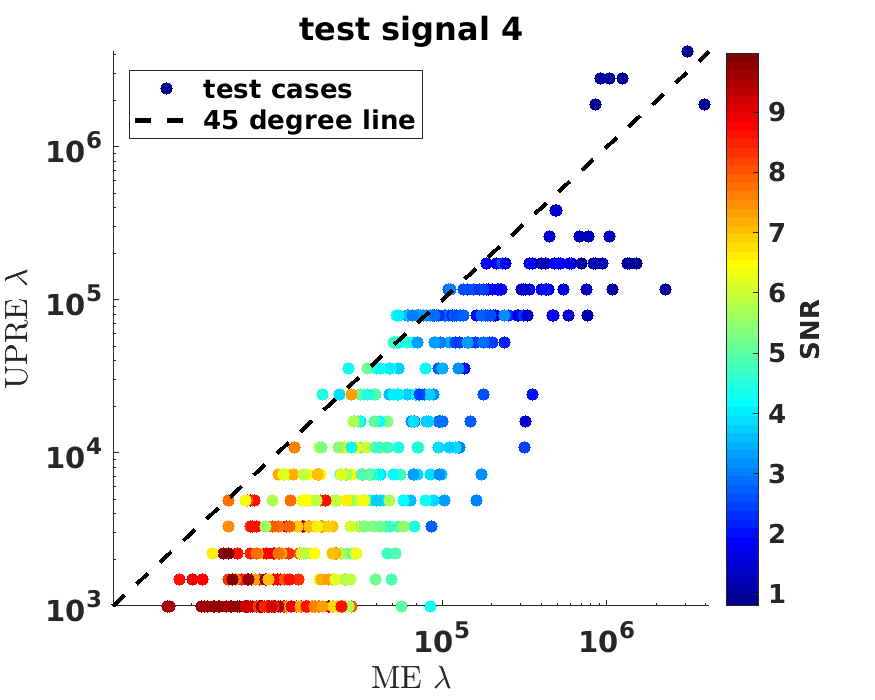}
 \caption{Recovered parameter from our method (with $\sigma$ assumed unknown) compared with UPRE (with $\sigma$ assumed known) for 4 test signals and 500 trials each.}
 \label{fig: UPRE}
\end{figure}

\begin{figure}[ht]
 \centering
 \includegraphics[width=.4\textwidth]{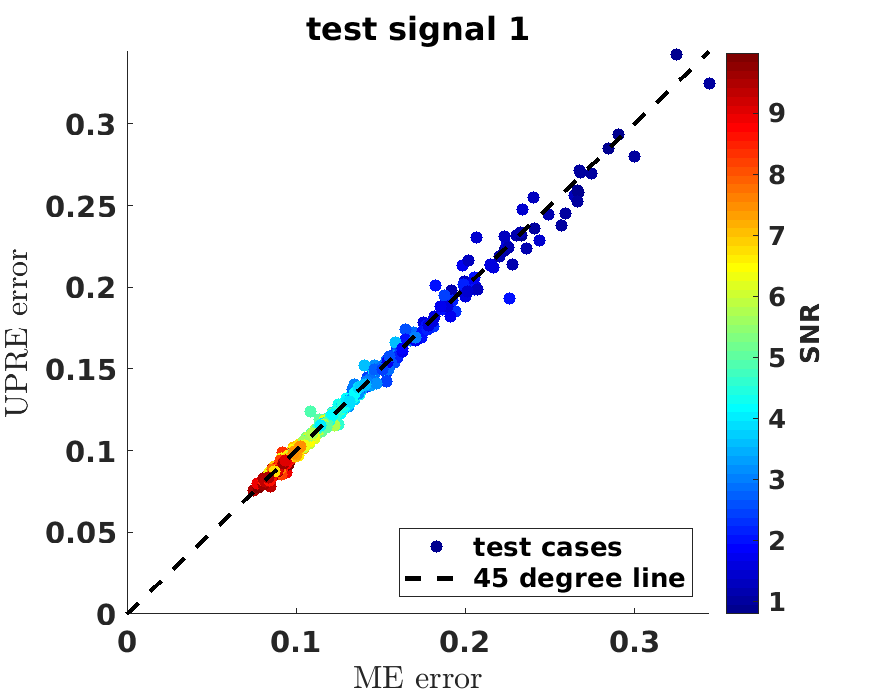}
 \includegraphics[width=.4\textwidth]{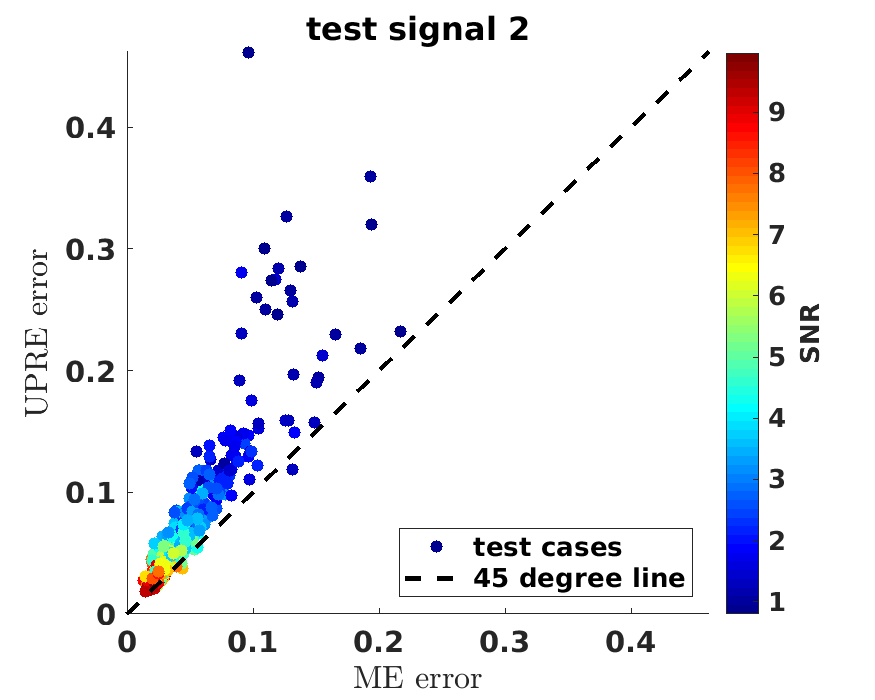} \\ 
 \includegraphics[width=.4\textwidth]{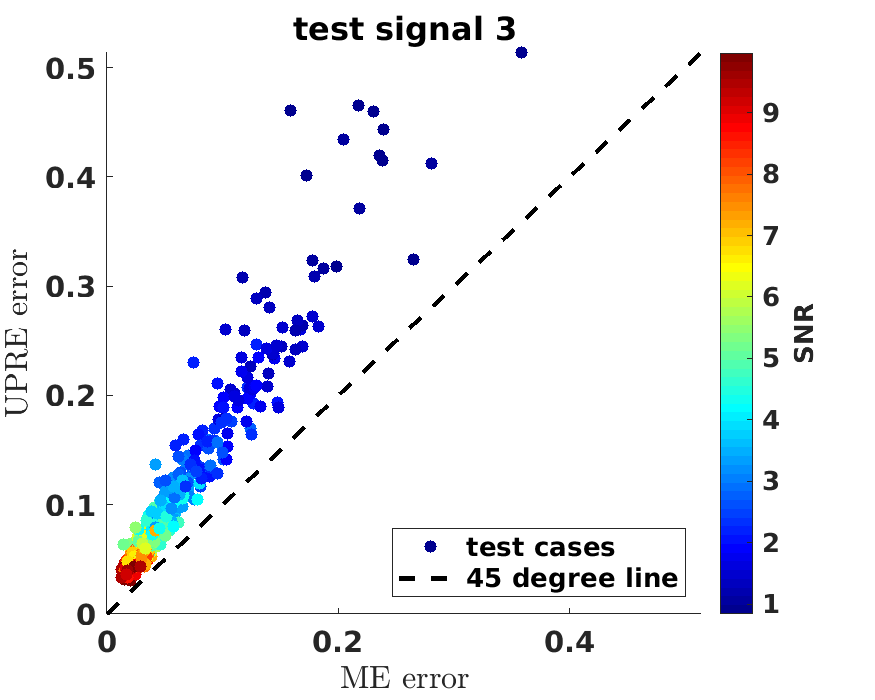}
  \includegraphics[width=.4\textwidth]{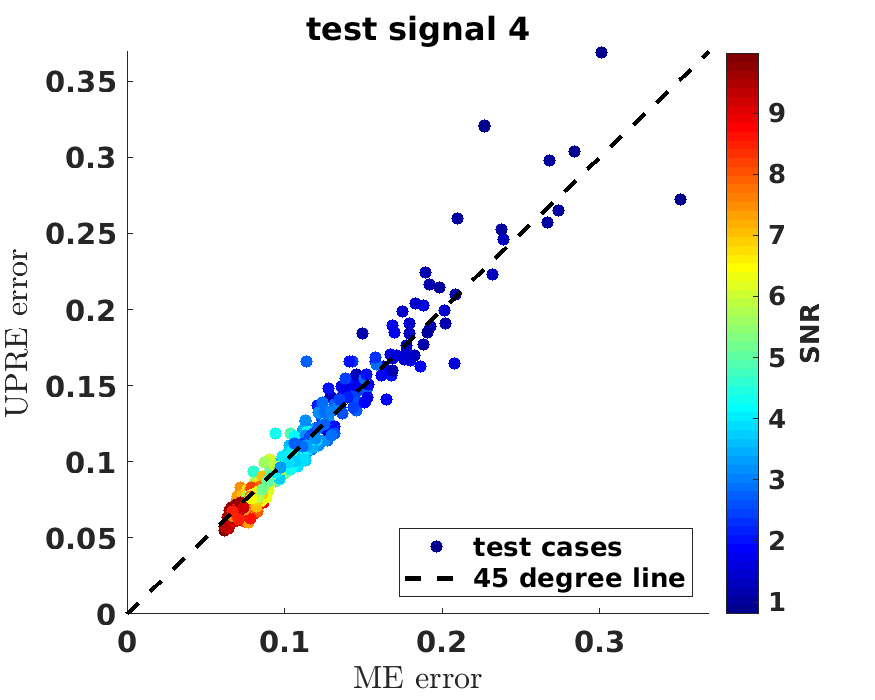}
 \caption{Error between the true solution and the reconstruction from our method (with $\sigma$ assumed unknown) compared with UPRE (with $\sigma$ assumed known) for 4 test signals and 500 trials each.}
 \label{fig: UPRE-error}
\end{figure}
In this section we compare our approach with UPRE.  
For 4 distinct test signals, 500 such simulations were performed.  In each case a random square sampling matrix $A\in\R^{n\times n}$ was generated randomly with independent normally distributed entries, and a random $\sigma$ value was generated on the interval $[.8,10]$.  Then our iterative scheme was implemented with a maximum of only 15 iterations to find $\sigma$ and $\eta$.  The 4 distinct signals are as follows: test signal 1 is a piecewise constant boxcar signal, test signal 2 is a piecewise linear hat function, test signal 3 is one period of a sine wave, and test signal 4 is the piecewise quadratic from \cite{sanders2017multiscale} and the previous example.  The corresponding regularization operators $T$ were chosen as various finite difference operators that appropriately match the signal properties, {\TLS{except again to avoid the inverse crime of overfitting each problem we used a first order finite difference regularizer for test signal 4.  For test signal one, we used a first order finite difference regularization, and for test signals 2 and 3 we used a multiscale second order finite difference regularization operator \cite{sanders2017multiscale}.}} Recall our method does not require knowledge of $\sigma$ whereas UPRE does require $\sigma$.


For UPRE, we followed the general approach suggested by Vogel \cite{vogel2002computational}, by selecting a series of $\lambda$ test values (e.g. 20) that are logarithmically equally spaced and choosing the parameter which minimizes the UPRE objective function.  From here, we even performed a second refinement by testing an additional series of parameter values near this parameter (also logarithmically equally spaced) and then settling on the optimal parameter from this set.

The scatter plots comparing the recovered parameters from our method and UPRE are shown in Figure \ref{fig: UPRE}.  These show that the two methods generally yield similar results (points near the dashed line mean similar recovered parameter values), with varying success and trends between the test signals.  These plots also show a very positive relationship between the SNR and $\lambda$ (lower SNR values almost always yield larger $\lambda$).

There are some notable discrepancies in the recovered $\lambda$ from our method and UPRE, particularly with test signals 2 and 3.  Therefore to discern the two methods we provide in Figure \ref{fig: UPRE-error} a second set of scatter plots from the simulations, where the resulting errors from our method and UPRE are plotted against one another.  For test signal 1, we see that the two methods generate similar approximations to the true solutions, which should be expected since the two generate very similar $\lambda$ values seen in Figure \ref{fig: UPRE}.  On the other hand, for the remaining 3 test signals some discrepancies were observed in the recovered $\lambda$ value, and the error plots in Figure \ref{fig: UPRE-error} show that our method tended to yield improved solutions (and hence $\lambda$ values), particularly for signals 2 and 3.  Therefore we conclude from these examples that our method, while assuming far less information that UPRE, may still yield improved results in the reconstruction.

\subsection{Accuracy of $\sigma$}

\begin{figure}[ht]
 \centering
 \includegraphics[width=.4\textwidth]{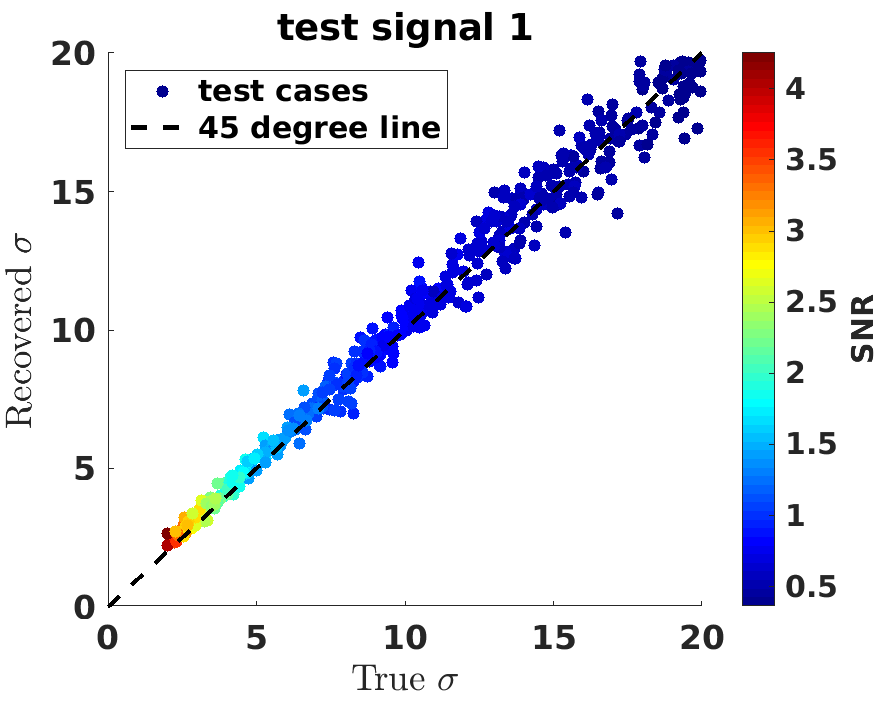}
 \includegraphics[width=.4\textwidth]{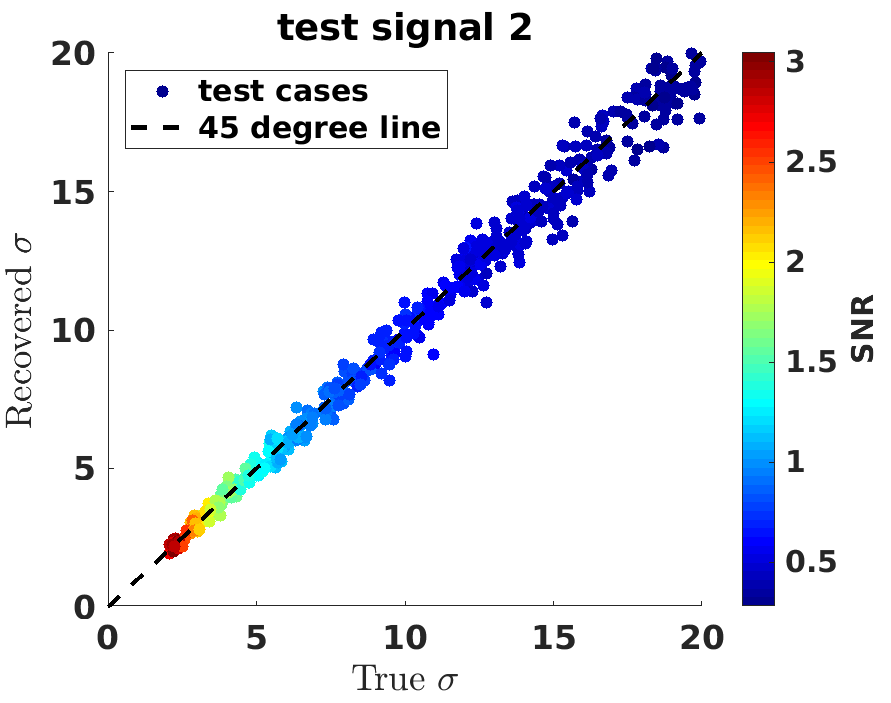} \\ 
 \includegraphics[width=.4\textwidth]{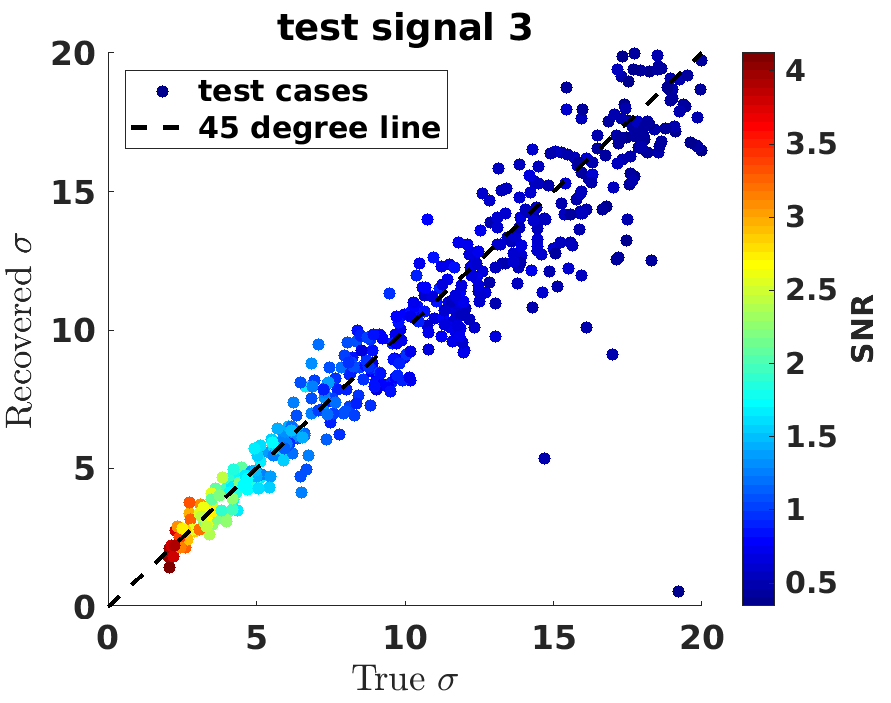}
  \includegraphics[width=.4\textwidth]{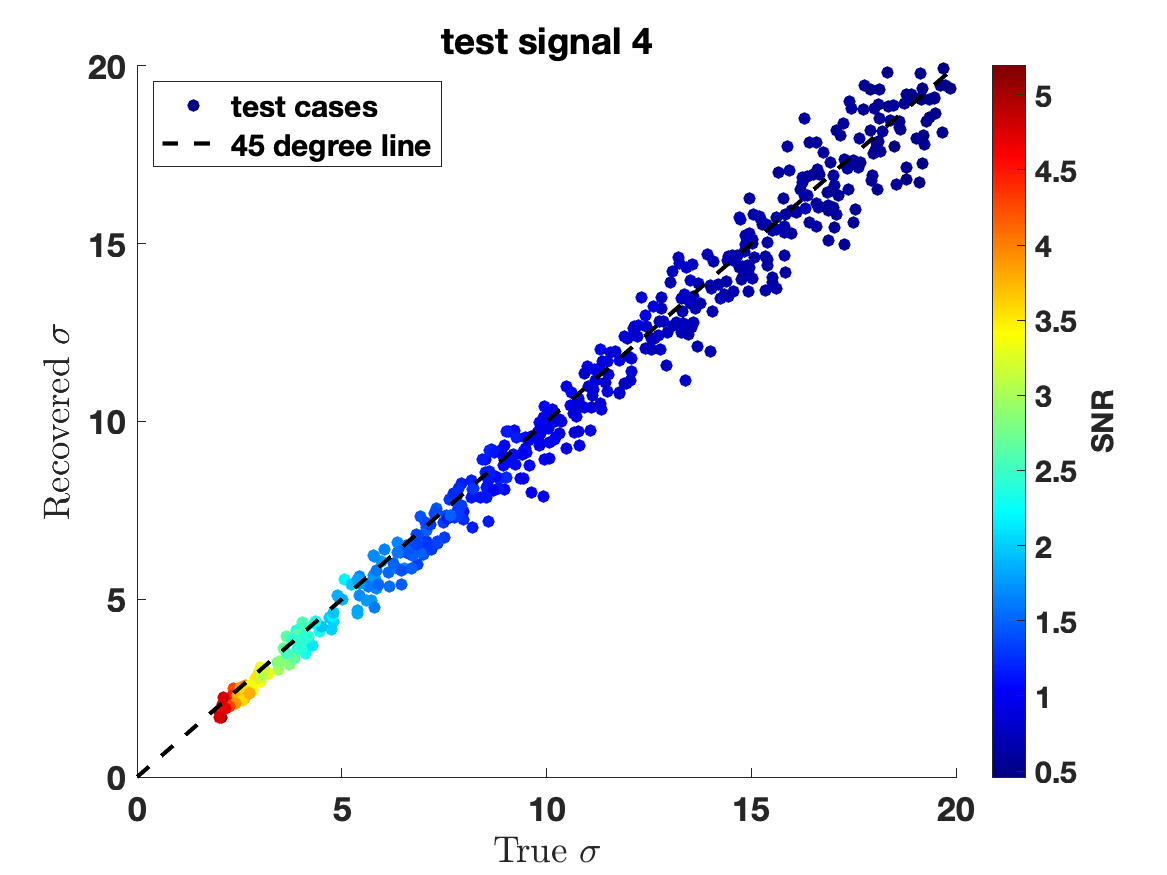}
 \caption{Recovered standard deviations plotted against the true standard deviations for 4 test signals and 500 trials each.}
 \label{fig: scatter}
\end{figure}


In this section we repeat a similar set of simulations from section \ref{sec: UPRE} to search for the parameter $\lambda = \sigma^2 / \eta^2$.  The general set up was the same as before, where this time we specified a random value of the SNR selected from the interval $[2,20]$ to determine the additive noise.

The scatter plots comparing the true $\sigma$ and recovered $\sigma$ are given in Figure \ref{fig: scatter}, and the corresponding SNR is given by the coloring.  These show that our scheme generally yields very accurate estimates of the variance under these settings.  There is greater spread for larger variances, but the \emph{relative} error does not necessarily increase.  Moreover, out of these 2,000 simulations there are only two clear poor solutions for $\sigma$, which come from test signal 3.

\subsection{2D Tomographic Example}\label{sec: tomo}
\begin{figure}[ht]
\centering
 \includegraphics[width=1\textwidth]{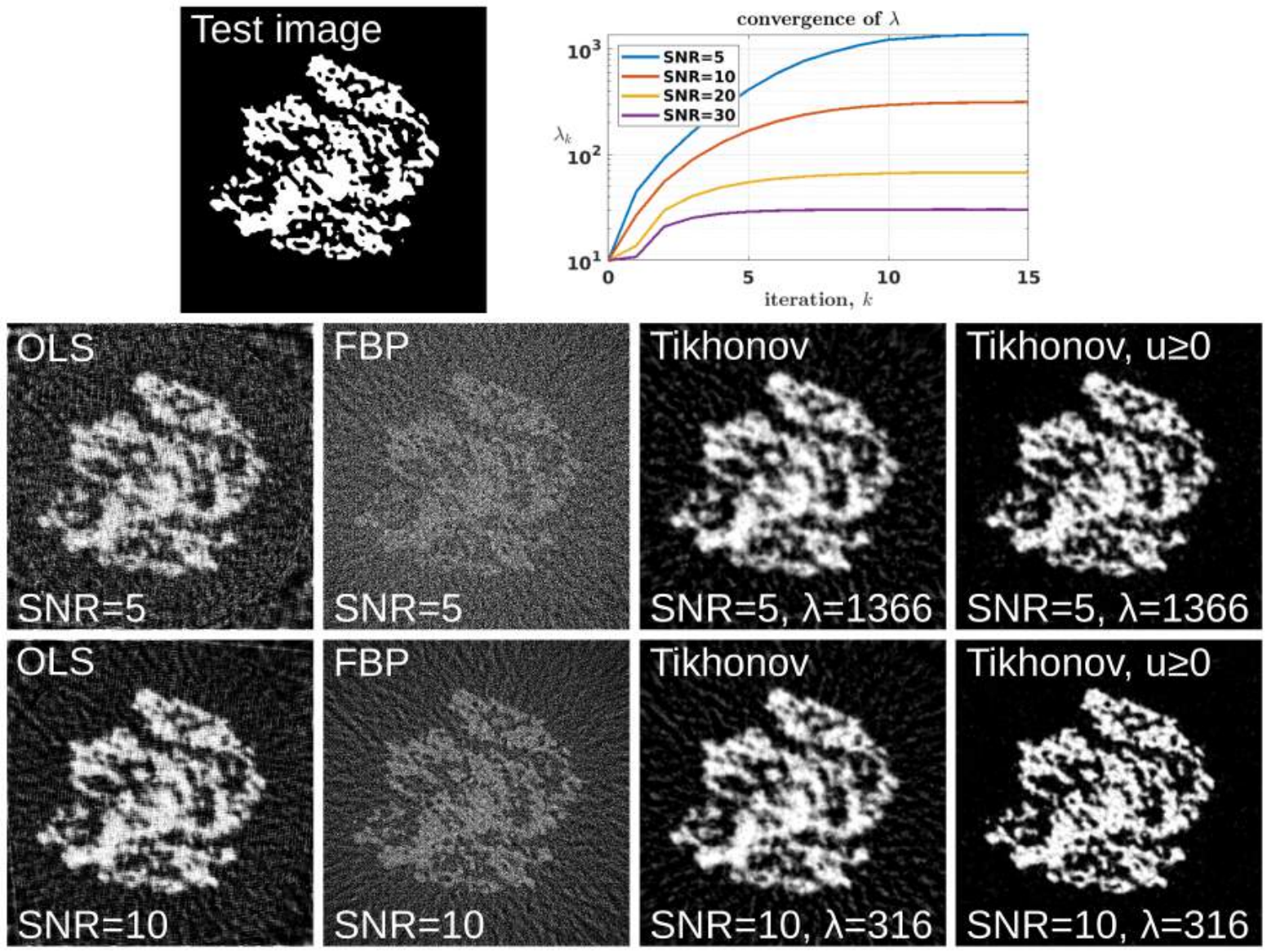}
 \caption{Tomographic reconstruction from 36 projection angles using our automatic parameter selection.}
 \label{fig: tomo}
\end{figure}
In this section we apply our automatic parameter selection to a 2D tomographic imaging example with parallel beam geometry.  For this problem the data vector $b$ takes values of the form
\begin{equation}\label{eq: tomo}
Au (t_i ,\theta_j ) = \int_{\R^2} u(x,y) \delta( t_i - (x,y) \cdot (\cos \theta_j , \sin \theta_j ) ) \, \rmd x \, \rmd y
\end{equation}
for $i=1,\dots, N$ and $j = 1,\dots , M$.  These types of inverse problems occur in a large number of applications, e.g. X-ray medical imaging, and electron and neutron tomography.  Typically the mesh formed by the $t_i$ is very fine, whereas the angular spacing $\theta_j$ can be quite large depending on the application.  

For our test problem we set a fixed $ \Delta \theta = \theta_{j+1} - \theta_j$ to be $5\degree$, and acquired data of the form (\ref{eq: tomo}) at each such angle increment over the full $180\degree$ extent for a total of 36 angles.  The spacing $t_i$ is set so that $N = 512$ resulting in $b\in \R^{512\cdot 36}$.  However, the image is also $u \in \R^{512\times 512}$, making the problem severely underdetermined.  Based on the parameters specified above, the corresponding forward and adjoint operator is computed efficiently in MATLAB with a sparse matrix, which can be constructed from our openly available software \cite{toby-web}. 

We tested 4 different SNRs, 5, 10, 20, and 30, and used a first order finite difference regularizer.  The initial choice of $\lambda$ was $\lambda_0 = 1$.  The resulting convergence of $\lambda$ for each case is plotted in logarithmic scale in Figure \ref{fig: tomo}, along with the resulting image reconstruction at SNRs of 5 and 10.   In the left two columns are the unregularized reconstructions from ordinary least squares (OLS) and filtered backprojection (FBP).  In the right two columns are the $\ell_2$ Tikhonov regularized solutions from our recovered optimal parameter, where the right most solution was recomputed under the constraint that $u$ be a nonnegative mass-density function.  The unconstrained Tikhonov and OLS solutions were computed with an iterative conjugate gradient method, which is likely the most efficient approach due to the sparse nature of the sampling operator.  The nonnegative Tikhonov solution is computed with a projected gradient decent approach \cite{toby-web}.

Observe that the quality of the regularized solutions are quite good compared with the unregularized, and the density constraint provides further notable improvement.  Also observe in plots of $\lambda$ that the recovered value accurately reflects changes in the SNR levels.

\section{Mapping Tikhonov Parameters onto L1 Parameters}\label{sec: L1}
Consider again the general inverse problem where the noise vector $\epsilon$ is i.i.d. mean zero Gaussian with variance $\sigma^2$ and the signal variance is $\eta^2$.  For the Gaussian prior on the signal, this lead us to the prior given in (\ref{prior}), and equivalently the regularization in (\ref{reg-l1}) with $p=2$ and $\lambda = \sigma^2/\eta^2$.  If we instead assume a Laplacian prior, i.e. an $\ell_1$ regularization with $p=1$ in (\ref{reg-l1}), then this prior with variance $\eta^2$ is given by \begin{equation}\label{prior-lap}
 p(u| \eta ) =\frac{\det T }{(\sqrt{2}  \eta)^{n}} \text{exp} \left(-\frac{\| T u\|_1}{\eta/\sqrt{2} }\right) .
\end{equation}
This leads to the MAP solution given by
\begin{equation}\label{eq: l1map}
u_{\sigma, \eta} = \arg \max_u C \text{exp} \left(-\frac{\| A u - b \|_2^2}{2\sigma^2} \right)
\text{exp} \left(-\frac{\| T u\|_1}{\eta/\sqrt{2} }\right),
\end{equation}
and we see this formulation is equivalent to (\ref{reg-l1}) with $p=1$ and $\lambda  = 2^{3/2} \sigma^2 / \eta$.  Due to the favorable analytical properties $\ell_2$ norm, in section (\ref{sec: alg}) we were able to derive an iteration for the variances $\sigma^2$ and $\eta^2$ with a Gaussian prior.  Moreover, the iterations for the $\ell_2$ prior can be evaluated very quickly, hence we use the iterative scheme for the Gaussian prior to find $\sigma^2$ and $\eta^2$.  These variances can then be put into the $\ell_1$ MAP estimation (\ref{eq: l1map}) and yield the corresponding $\lambda$ for the $\ell_1$ regularized problem as written above.  The $\ell_1$ optimization problem is solved using the alternating direction method of multipliers approach \cite{wu2010augmented,boyd2011distributed}, and our code is openly available \cite{toby-web}.

\begin{figure}[ht]
 \centering
 \includegraphics[width=.4\textwidth]{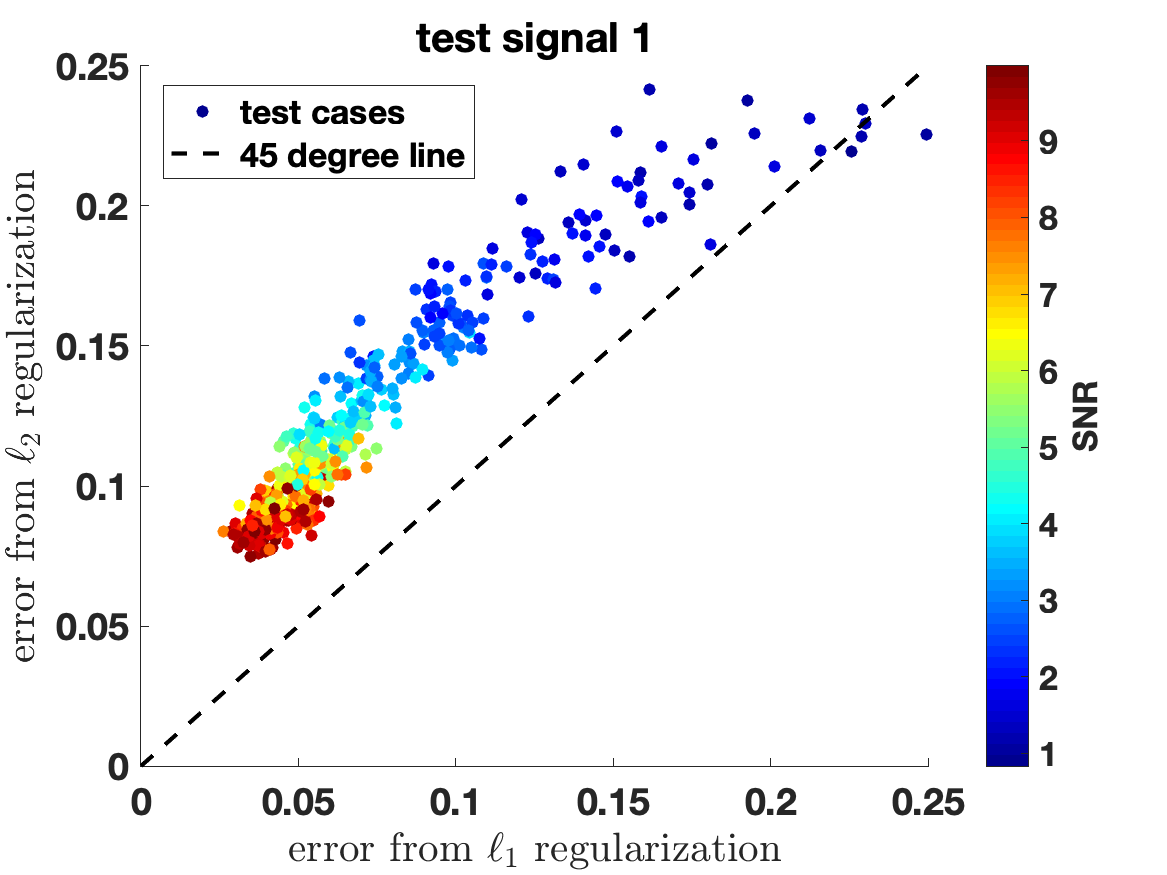}
 \includegraphics[width=.4\textwidth]{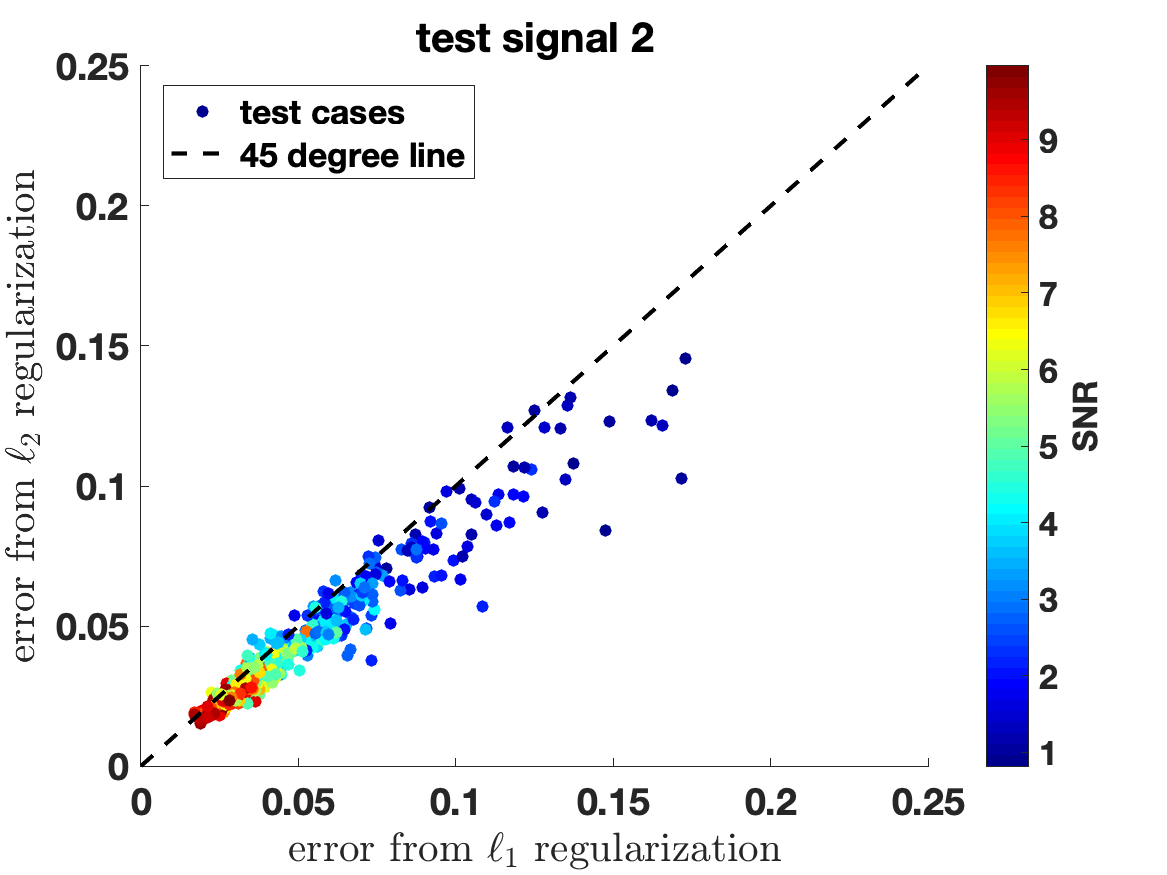} \\ 
 \includegraphics[width=.4\textwidth]{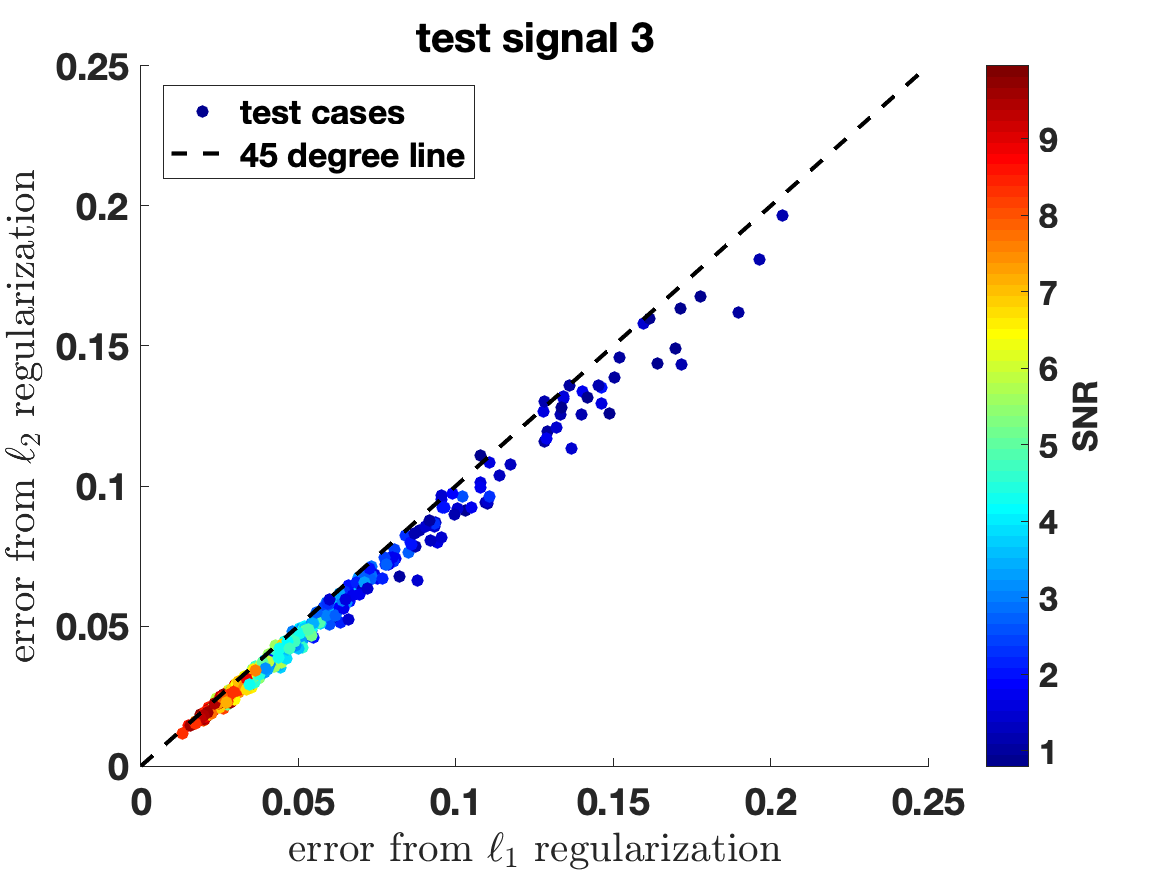}
  \includegraphics[width=.4\textwidth]{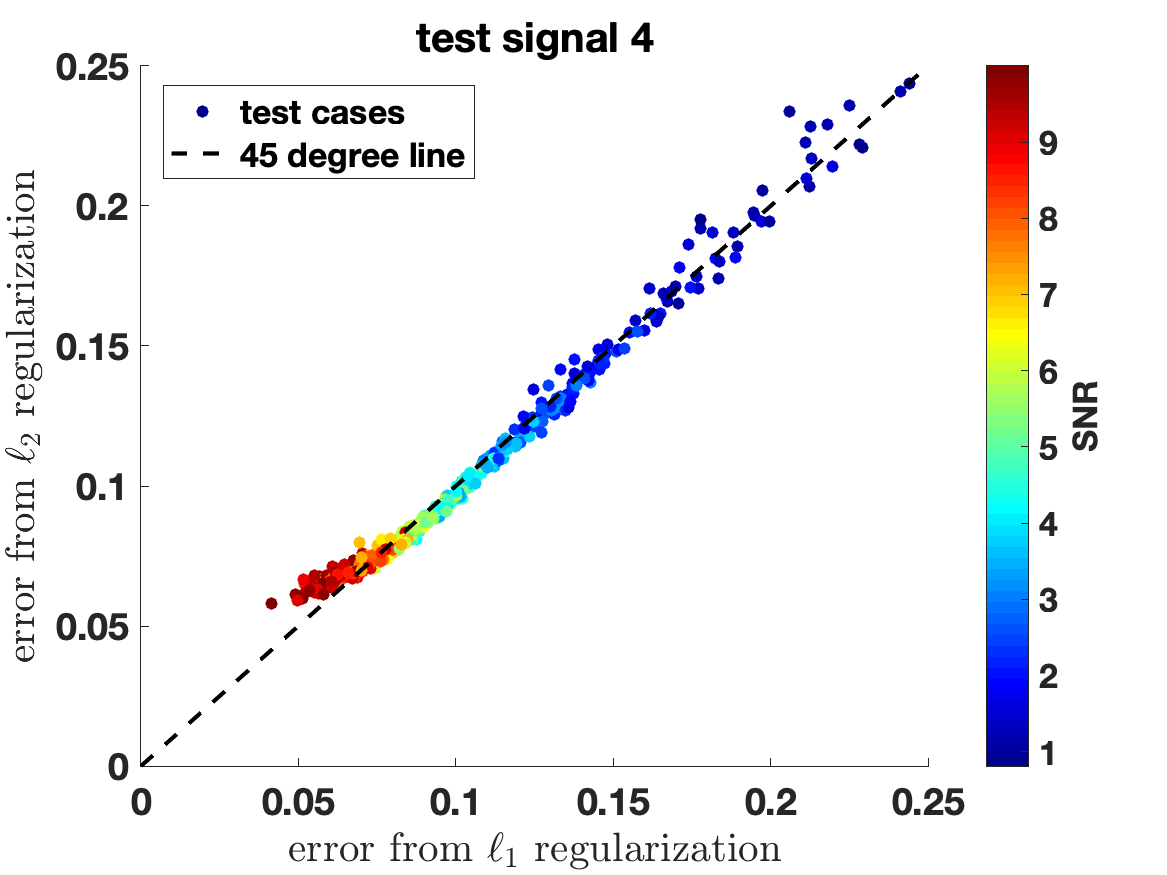}
 \caption{Scatter plot for the $\ell_1$ and $\ell_2$ regularization errors between the recovered and true signals from the recovered optimal parameter $\lambda$ for the 4 test signals and 500 trials each.  The optimal $\ell_2$ parameter was determined by ML and was projected onto the $\ell_1$ parameter using the Bayesian formulation.}
 \label{fig: L1}
\end{figure}

The use of the Gaussian prior to select $\lambda$
for the Laplace prior is justified as follows.  The prior $p(u|\eta)$ given by Eq. (\ref{prior-lap}), as a function of $Tu$,
is the density function of a sum of $n$ independent
Laplace random variables each with variance $\eta^2$.
This can be approximated by a sum of Gaussian random variables
having the same variance,
and by the central limit theorem the difference between the two distributions
gets smaller as $n$ becomes larger.
Moreover, the marginal probability $p(b) = p(b|\sigma,\eta)$,
defined in the proof of Lemma~2,
depends not pointwise on the prior but on an integral of the prior,
further diminishing the difference.
And it is the maximization of $p(b|\sigma,\eta)$
that determines $\sigma$ and $\eta$.
Therefore, the optimal choice for $\ell_2$
is expected to be very nearly optimal for $\ell_1$.

Numerical tests of these concepts are presented in Figure (\ref{fig: L1}).  Here we have solved for $\sigma$ and $\eta$ for the Tikhonov regularized problem as before, and then used these parameters to yield $\lambda$ for the $\ell_1$ regularized problem as written above.  From these parameters, the measured errors between the true solution and the recovered solution are plotted against one another for the $\ell_1$ and $\ell_2$ regularized solutions.  The set up for these simulations was the same as those in section \ref{sec: UPRE}.  It is observed from these plots that projecting the parameters $\sigma$ and $\eta$ onto the $\ell_1$ problem generally provides solutions that are comparable to the $\ell_2$, and in many cases the $\ell_1$ solution provides better results.  Hence it appears to be a quite reasonable strategy to use the ML estimation for the $\ell_2$ parameter to select the $\ell_1$ parameter.


{\TLS{Finally, we present numerical results for deconvolution of 2D images with different noise levels and convolutional kernels. The convolution kernels are symmetric Gaussian point spread functions with different standard deviations, $\omega$, varying between $\omega =  0.5$ and $\omega = 3$, where we are using the convention that a unit length in an image is one pixel.  The noise added is i.i.d. white Gaussian noise, where the SNR is defined by the mean image value divided by the standard deviation of the noise.  The simulations are performed on three classical test images, the cameraman, Shepp-Logan phantom, and the peppers image.  The first two images are grayscale and the third was converted to grayscale for these simulations.  An extremely fast version of the algorithm is implemented for deconvolution, as outlined in the next section.  In each case, we simply used a first order finite difference regularization matrix, $T$.  The $\ell_2$ parameters found from our algorithm are again projected onto the $\ell_1$ formulation.  Some of the resulting images for the cameraman are shown Figure \ref{fig: cameraman}.  

From these simulations, the resulting errors (compared with the known true solution) are shown in Figure \ref{fig: numerical-deconv}, where the relative error presented is analogous to (\ref{eq: convergence}).  These results indicate that our approach and UPRE yield very similar results for these simulations, with marginal improvements in our approach for a few cases.  Moreover, our approach, as opposed to UPRE, was implemented without the knowledge of the noise level in the data and uses a quickly convergent fixed point iteration to find the optimal parameters.  The $\ell_1$ solutions found by projecting the recovered $\ell_2$ parameters are notably improved in each case, and significantly so for the Shepp-Logan phantom (results in middle row), likely due to the fact that the total variation $\ell_1$ regularization is ideal for this image.  

The plots in the far right column provide a closer comparison by showing the errors curves resulting from the fourth row of each error matrix, which is the case $\omega = 1.33$.  Included in these plots is the error obtained from the $\ell_2$ and $\ell_1$ solutions where the parameter was tuned to minimize the error between the reconstruction and the true solution, which is obviously not possible in practice.  Nonetheless, this result indicates that the solution obtained using our approach for both the $\ell_2$ and $\ell_1$ parameter selection is \emph{nearly} optimal, particularly for the last two test images.  One could not hope for a better result in this case.
}}

\begin{figure}[ht]
\centering
\includegraphics[width=1\textwidth]{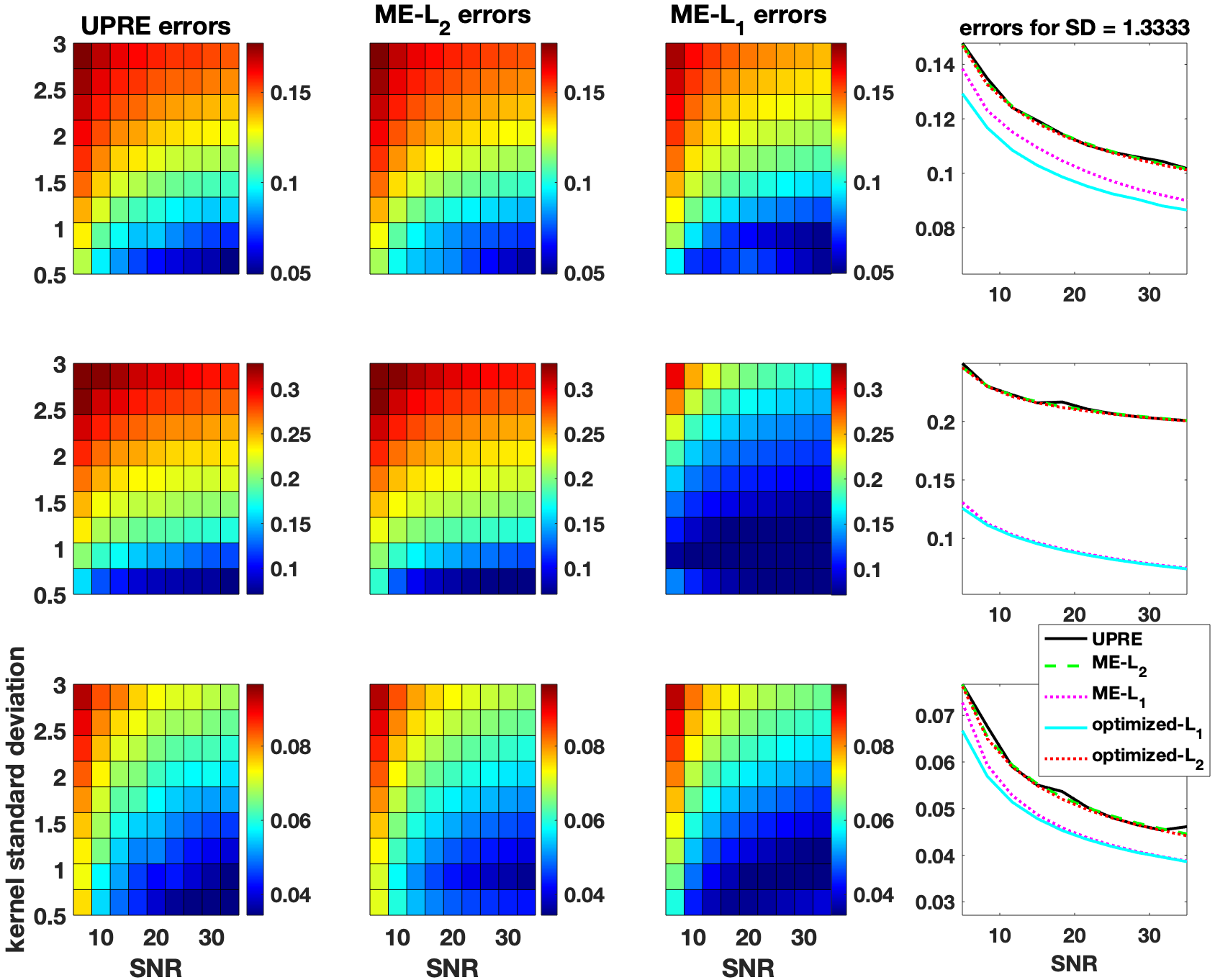}
\caption{Reconstruction errors for 2D image deconvolution examples.  Top row is for the cameraman image, middle row is for the Shepp-Logan phantom, and the bottom row is for the peppers image.  The plots in the far right column show the error from the fourth row error matrices where the standard deviation of the blurring kernel is $\omega = 1.33$.  The "optimized" solutions in these plots refer to the result of tuning $\lambda$ to the best possible solution by making use of the true solution, which is obviously an absurd baseline, but it indicates that our proposed parameter selection is \emph{near} optimal.}
\label{fig: numerical-deconv}
\end{figure}

\begin{figure}[ht]
\centering
 \includegraphics[width=1\textwidth]{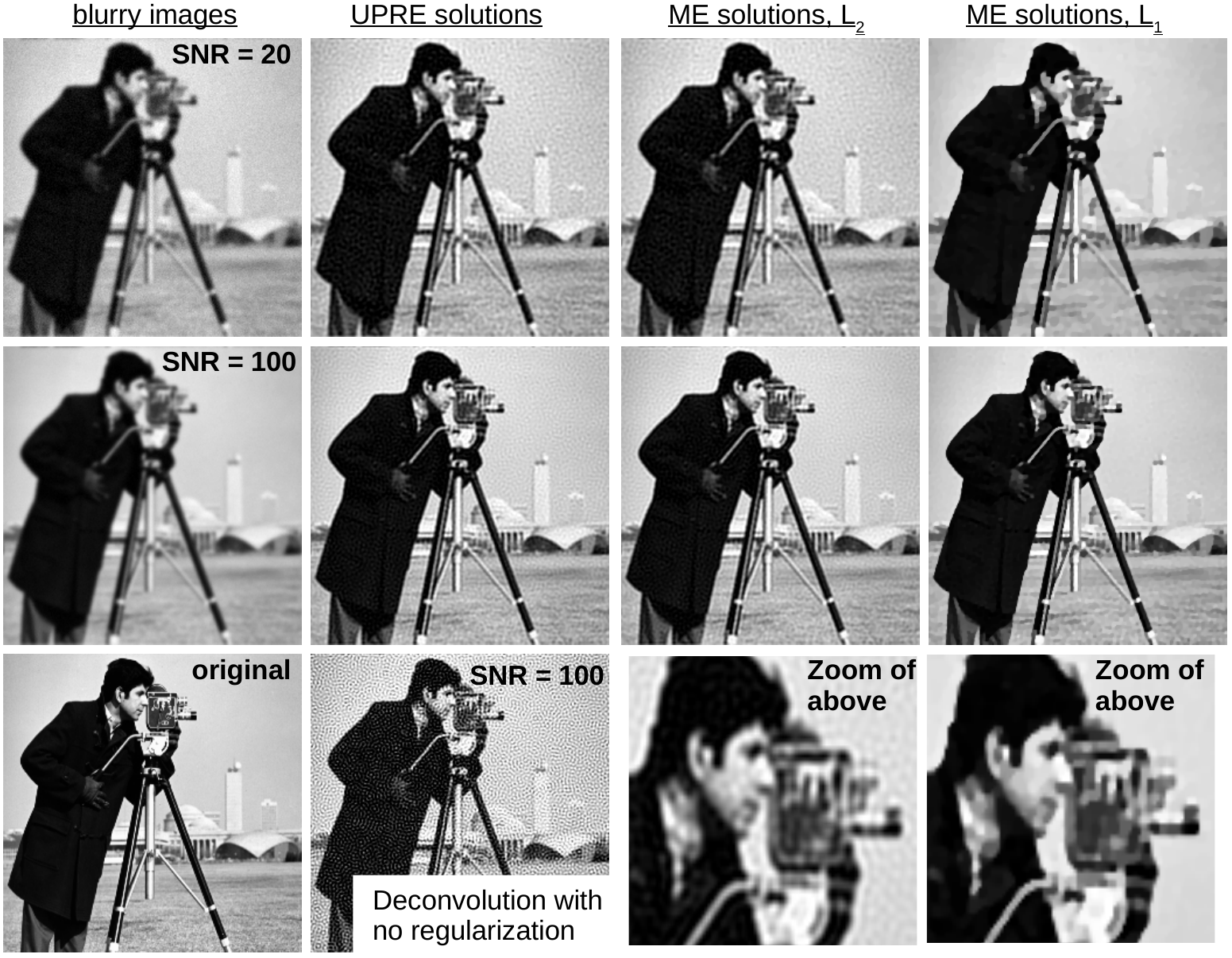}
 \caption{Comparisons of deconvolution with $\ell_2$ regularizations and the projection of the parameters onto the $\ell_1$ regularization.  The standard deviation of the Gaussian point spread function for these images is $\omega = 1.5$.}
 \label{fig: cameraman}
\end{figure}

\FloatBarrier
\section{Accelerated Iterations for Denoising and Deconvolution}\label{sec: acc}
In this section we show how to significantly reduce the computational load for our iterative scheme for the popular denoising and deconvolution applications, and in the appendix these derivations are easily extended to Fourier reconstruction problems (e.g. SAR and MRI).  It is achieved by obtaining exact formulas for the traces appearing in (\ref{eq: exp1}) and (\ref{eq: exp2}) and the solutions to (\ref{eq: tik}) that require only one inverse FFT and multiplication by a diagonal matrix in each iteration. {\TLS{Moreover, the FFT may be avoided if one is satisfied with the leaving the iterated solutions in the Fourier domain, leaving us with only $O(n)$ flop count for each iteration.}} The derivations needed for 1D are provided here, and extensions to higher dimensions and Fourier sampling are provided in the appendix.

The case of denoising occurs when the sampling matrix $A$ is the identity and the data vector is a noisy version of the image given by say $\tilde u$.  More generally, we consider the deconvolution problem, where the sampling matrix is circulant.  These denoising and deconvolution problems are written as
\begin{equation}\label{eq: tik-deconv}
 \min_{u} \|C u - \tilde u\|_2^2 + \lambda \| Tu \|_2^2,
\end{equation}
where $C$ is a circulant matrix and $\tilde u$ is a noisy and/or blurred version of $u$. 

In what follows, formulas for the traces needed for (\ref{eq: iter1}) and (\ref{eq: iter2}) are derived analytically by determining and summing the eigenvalues.  First observe that circulant matrices are diagonalizable by the unitary discrete Fourier transform\footnote{This can be seen as a direct result of the Fourier convolution theorem.} denoted by $\mathcal F$, and hence for this calculation we consider circulant regularization matrices $T = T_r$ of the form
\begin{equation}\label{1D-matrices}
 T_1 = \left(
 \begin{array}{ccccc}
  -1 & 1 & 0 & \dots & 0\\
  0 & -1 & 1 & \dots & 0\\
  \vdots & & \ddots & & \vdots\\
  1 & 0 & 0 & \dots & -1
 \end{array}
 \right) 
 ~ \text{and} \quad
 T_2 = \left(
 \begin{array}{ccccc}
  1 & -2 & 1 & \dots & 0\\
  0 & 1 & -2 & \dots & 0\\
  \vdots & & \ddots & & \vdots\\
  -2 & 1 & 0 & \dots & 1
 \end{array}
 \right) .
\end{equation}
Notice in general we have $T_r = T_1^r$.  
We write the diagonalization of $T_r$ and $C$ as 
\begin{equation}\label{new-diag}
\begin{split}
T_r & = \mathcal F^{-1} \Lambda_r \mathcal F\\
C & = \mathcal F^{-1} \Lambda_C \mathcal F,
\end{split}
\end{equation}
where $\Lambda_r$ and $\Lambda_C$ contain the eigenvalues of $T_r$ and $C$ respectively.   The eigenvalues in the diagonal matrices $\Lambda_r$ and $\Lambda_C$ may be evaluated by simply taking the discrete Fourier transforms (DFT) of the first column of $T_r\T$ and $C\T$, respectively.

{\TLS{These diagonalized forms allow us to analytically derive simple formulas for the traces and evaluate them very cheaply.}} Using (\ref{new-diag}), observe the expression for the matrix products within the traces in (\ref{eq: iter1}) and (\ref{eq: iter2}) as
\begin{equation}\label{eq: trace1}
\begin{split}
 H^{-1} C\T C & = \mathcal F^{-1} ( |\Lambda_C|^2 + \lambda |\Lambda_r|^2)^{-1} |\Lambda_C|^2 \mathcal F \\\
  H^{-1} T_r\T  T_r & = \mathcal F^{-1} ( |\Lambda_C|^2 + \lambda |\Lambda_r|^2)^{-1} |\Lambda_r|^2 \mathcal F .
\end{split}
\end{equation}
Therefore, $H^{-1} C\T C$ and $H^{-1} T_r\T T_r$ are also diagonalized by the Fourier transform and also circulant, and their eigenvalues are given by the elements of the diagonal matrices $( |\Lambda_C|^2 + \lambda |\Lambda_r|^2)^{-1} |\Lambda_C|^2 $ and $( |\Lambda_C|^2 + \lambda |\Lambda_r|^2)^{-1} |\Lambda_r|^2 $ respectively.  The eigenvalues for $C$ will depend on the specific convolution operator, and in the case of denoising these eigenvalues are all obviously 1.  In general, we denote these eigenvalues by $\gamma_j(C) = \sum_{k=1}^n c_{1,k} \rme^{-\rmi 2\pi j (k-1) /n}$, for $j=0,1,\dots , n-1$.  For $T_r$, we have the exact expression for these eigenvalues as
\begin{equation}\label{eq: eig1d}
\gamma_{j}(T_r) =  \gamma_{j}(T_1^r) 
= (\e^{-\rmi 2\pi j/n} -1)^r,
\end{equation}
for $j=0,1,\dots,n-1$, and therefore $|\gamma_j(T_r) |^2 = 4^r \sin^{2r} (\pi j/n)$.
 Then the traces needed for our algorithm are given by
\begin{align}
 \text{trace}(H^{-1} C\T C ) & = \sum_{j=0}^{n-1} \left( |\gamma_j(C)|^2 + \lambda 4^r \sin^{2r}(\pi j/n)\right)^{-1} |\gamma_j(C)|^2  \label{eq: trace3} \\
  \text{trace}(H^{-1} T_r\T T_r ) & = \sum_{j=0}^{n-1} \left( |\gamma_j(C)|^2 + \lambda 4^r \sin^{2r}(\pi j/n)\right)^{-1} 4^r \sin^{2r}(\pi j/n) \label{eq: trace4},
\end{align}
which are reevaluated very cheaply for each update on $\lambda$.

{\TLS{This handles the traces needed for (\ref{eq: iter1}) and (\ref{eq: iter2}), and next we need to efficiently evaluate the numerators, $\| Cu_\lambda -\tilde u \|_2^2$ and $\| T_r u_\lambda \|_2^2$.  Once again using the diagonalized form of $T_r$, observe the latter norm may be written as
\begin{equation}\label{new-norm}
\| Tu_\lambda \|_2^2 = u_\lambda\T \mathcal F^{-1} | \Lambda_r |^2 \mathcal F u_\lambda .
\end{equation}
Leveraging Parceval's theorem with the diagonalized form of $C$ into the first norm leads to
\begin{equation}\label{new-norm2}
\begin{split}
\| C u_\lambda - \tilde u \|_2^2 
&= \| \mathcal F (Cu_\lambda - \tilde u )\|_2^2 \\
& = \| \Lambda_C \mathcal Fu_\lambda - \mathcal F \tilde u \|_2^2 .
\end{split}
\end{equation}
From (\ref{new-norm}) and (\ref{new-norm2}), observe that with the availability of the Fourier transform of $u_\lambda$ these norms may be computed by multiplication with a diagonal matrix followed by a dot product.  Using the diagonalized representations once more the solution to (\ref{eq: tik-deconv}) is written as
\begin{equation}\label{eq: tik-exact}
u_\lambda = \mathcal F^{-1} D_{C,r} \F \tilde u ,
\end{equation}
where $D_{C,r} = (|\Lambda_C|^2 + \lambda |\Lambda_r|^2)^{-1} \overline{\Lambda_C}$ is a diagonal matrix \footnote{Observe that (\ref{eq: tik-exact}) is just a particular form of a Wiener filter}.  Therefore the Fourier transform of $u_\lambda$ is obtained by only multiplying $\F \tilde u$ by a diagonal matrix, and $\F \tilde u $ only needs to be precomputed.  Moreover, the inverse Fourier transform in (\ref{eq: tik-exact}) needed to obtain $u_\lambda$ only needs to be evaluated following the iterations, so that only one initial FFT is needed to evaluate $\F \tilde u$ prior to the iterations and one final inverse FFT to evaluate the solution following the iterations. }}
A pseudo algorithm summarizing these ideas is provided in Algorithm \ref{algo-fast}.

\begin{algorithm}[!ht]
\caption{: Fast Automated Deconvolution}
\label{algo-fast}
\begin{algorithmic}[1]
\STATE{Inputs: $\tilde u$, $T_r$, $C$, $\lambda_0$.}
\STATE{Evaluate ${\tilde u_\F} = \F \tilde u$ using an FFT.}
\FOR{$k$=0 \TO K}
\STATE{Define $H_k = C\T C + \lambda_k T_r\T T_r$.}
\STATE{Evaluate the Fourier transform of the solution requiring only multiplication by a diagonal matrix: $\F u_k^\ast =  (| \Lambda _C|^2 + \lambda_k |\Lambda_r|^2)^{-1} \overline{\Lambda_C}  \tilde u_\F$}
\STATE{Evaluate $\text{trace}(H_k^{-1}C\T C)$ and $ \text{trace}(H_k^{-1}T_r\T T_r) $ by (\ref{eq: trace3}) and (\ref{eq: trace4}).}
\STATE{Evaluate $\| C u_k^* - \tilde u\|_2^2$ and $\| Tu_k^* \|_2^2$ using only the $\F u_k^*$ with (\ref{new-norm}) and (\ref{new-norm2}).}
\STATE{Set $\sigma_{k+1}^2 = \| C u_k^\ast - \tilde u \|_2^2/(n- \text{trace}(H_k^{-1} C\T C)) $ and $\eta_{k+1}^2 = \| T_r u_k^\ast \|_2^2 /(n- \lambda_k \text{trace}(H_k^{-1}T_r\T T_r) )$.}
\STATE{Set $\lambda_{k+1} = \sigma_{k+1}^2 / \eta_{k+1}^2$.}
\ENDFOR
\STATE{Obtain solution by evaluating inverse Fourier transform of $\F u_K^*$.}
\end{algorithmic}
\end{algorithm}

{\TLS{To formally point out the computational load of this algorithm, recall the forward operator $A\in \R^{m\times n}$, and for deconvolution $m=n$.  Then it can be observed that the main computational load for each main outer loop of the general algorithm presented in Algorithm \ref{algo-main} requires $O(n^2 N(J+1))$ flops, where $J$ is the number of random vectors used and $N$ is the number of iterations used in the CG solver.  Alternatively, the main computational load for each loop in Algorithm \ref{algo-fast} is given by dot products and multiplication by diagonal matrices, which are only $O(n)$ operations, providing massive reduction in computation compared with the general formulation.

Average run times (in MATLAB) for determining the optimal parameters for deconvolution using this fast approach are provided in table \ref{table1}, along with the time needed to then compute the optimal $\ell_1$ solution from the determined parameters using the state-of-the-art ADMM approach \cite{boyd2011distributed}.  These timed simulations were taken from the simulated results in Figures \ref{fig: numerical-deconv}.  The convergence for each simulation was based on the relative change in the iterated solution reducing to less than a fixed tolerance, $10^{-4}$.  Observe that the time required to determine the parameters is roughly 12 to 30 times faster than needed to compute the $\ell_1$ solution.  Considering this along with the plots in Figure \ref{fig: numerical-deconv}, this approach appears to very practical for either extremely fast optimized deconvolution with Tikhonov regularization, or for simply choosing the parameters for the superior $\ell_1$ model with practically no additional computation time.
}}


{\toby{The analytical extensions of these concepts to 2D problems and Fourier sampling are provided in the appendix.  Moreover, convergence analysis of algorithm 2 in the denoising case is given in section \ref{sec: conv}, with the detailed proofs provided in the appendix.}} 

\begin{table*}[h!]
\begin{center}
\begin{tabular}{ c | c | c | c | }
	 image name & pixel count & $\ell_2$ parameter time & $\ell_1$ reconstruction time \\
            \hline 
            cameraman  & 256x256 & 0.0391  & 0.4997 \\
            Shepp-Logan  & 256x256 & 0.0338  & 0.7862 \\
            peppers  & 384x512 & 0.0822  & 2.3706 \\
            \hline
\end{tabular}
\end{center}
\caption{Average computational run time (in seconds) to determine the optimal parameters for deconvolution and to then compute the corresponding $\ell_1$ reconstruction from the determined parameters.}
\label{table1}
\end{table*}

\FloatBarrier
\section{Convergence Analysis of Algorithm 2} \label{sec: conv}
In this section some asymptotic and convergence analysis of algorithm 2 is provided in the denoising case, $C = I$.  The detailed proofs of our claims are given in the appendix. 
\begin{prop} For a general circulant matrix $T$ with eigenvalues $\{ \gamma_j \}_{j=0}^{n-1}$ used for regularization, the fixed point iteration in Algorithm 2 can be represented by 
$$\lambda_{k+1} = f(\lambda_k; \tilde u,T),$$
where
\begin{equation}\label{eq:alg2A}
 f(\lambda; \tilde u,T) = \lambda \frac{\||\Lambda|^2 B(\lambda)\hat{u}\|_2^2}%
{\||\Lambda| B(\lambda)\hat{u}\|_2^2}
 \frac{\tr(B(\lambda))}{\tr(|\Lambda|^2 B(\lambda))}
\end{equation}
where $B(\lambda) = (I +\lambda|\Lambda|^2)^{-1}$ and $\hat u = \mathcal{F} \tilde u$.

\label{prop:alg2}
\end{prop}

The following proposition gives further insight on the behavior of $f$.

\begin{prop} Suppose $\tilde u \neq 0$. The fixed point iteration function $f$ defined in \eqref{eq:alg2A} satisfies:
\begin{enumerate}
\item For $T=I$, $f(\lambda,\tilde u,I) = \lambda$ and every real number is a fixed point.
\item Zero is a fixed point of $f$. Moreover, zero is a stable fixed point if 
$$
f'(0;\tilde u,T)=\frac{ n}{\sum_{j=0}^{n-1}  |\gamma_j|^2}\frac{\|T\T T \tilde u\|^2}{\|T \tilde u\|^2 } <1.
$$

\item For $T=T_r$ with $r>0$ as in (\ref{1D-matrices}), the asymptotic behavior of $f$ as $\lambda\to \infty$  is given by 
$$
f(\lambda;\tilde u, T_r) \sim  \lambda^2 \kappa_{\infty}(\tilde u,T_r), 
$$
$$
\kappa_{\infty}(\tilde u,T_r)  = \frac{ \sum_{j=1}^{n-1} |\hat{u}_j|^2 }{(n-1) \sum_{j=1}^{n-1} |\hat{u}_j|^2/ |\gamma_j|^2},
$$
and $ 4^r\sin^{2r}(\pi/n)/(n-1) \le \kappa_{\infty}(\tilde u,T_r)\le4^r/(n-1)$.
\end{enumerate}
\label{prop:PropertiesAlg2}
\end{prop}

Figure~\ref{fig:Alg2R2} illustrates features of $f$ in light of Proposition~\ref{prop:PropertiesAlg2}. Here the input signal was generated by adding Gaussian noise to the piecewise constant function shown in the right frame. In this example, SNR $= 5$ and second order ($r=2$) regularization was used. Notice that $f(\cdot, \tilde u,T_r)$ has 3 fixed points: $0$, $\lambda_{opt}\approx 16.5$, and $\lambda_\infty \approx 1/\kappa_{\infty}(\tilde u,T_r)$. As typically the case, the trivial fixed point $\lambda = 0$ is unstable -- a fact that can easily be verified using the second part of Proposition~\ref{prop:PropertiesAlg2}. A second unstable fixed point can be observed in the asymptotic regime $f(\lambda;\tilde u, T_r) \sim   \kappa_{\infty}(\tilde u,T_r)y^2$, since $1/\kappa_{\infty}(\tilde u,T_r)$ solves the equation $\kappa_{\infty}(\tilde u,T_r)\lambda^2 = \lambda$. Finally $\lambda_{opt}\approx 16.5$ is the only stable fixed point. The middle frame in Figure~\ref{fig:Alg2R2} shows that $f$ is a contraction in a neighborhood of this value. The reconstruction using $\lambda_{opt}$ is given in right frame of this figure.

\begin{figure}[ht]
\centerline{\includegraphics[width=16cm]{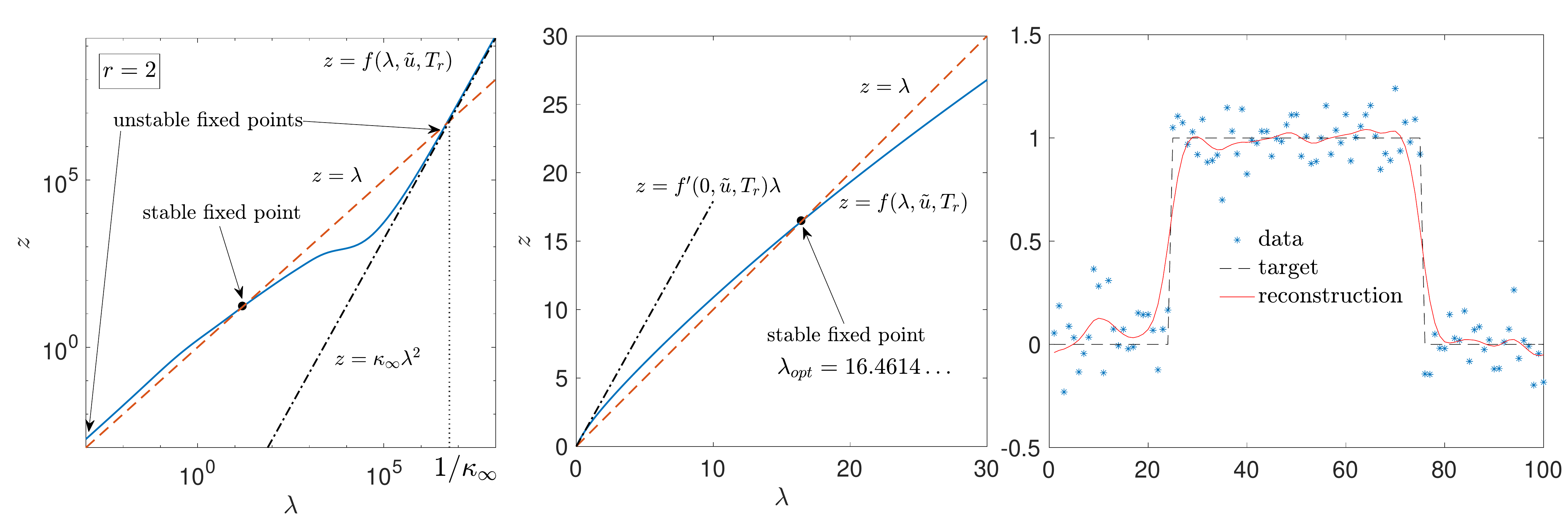}}
\caption{From left to right: a loglog plot of the fixed point iteration function for $r=2$; a linear plot of of the same function showing the fixed points 0 and $\lambda_{opt}$; and plots of the corresponding data $\tilde u$,  the target reconstruction, and the denoised signal using $\lambda_{opt}$.}
\label{fig:Alg2R2}
\end{figure}

We point out that there are instances in which $\lambda =0$ is a stable fixed point. According to Proposition~\ref{prop:PropertiesAlg2}, this happens when 
$$
\frac{\|T_r\T T\tilde u\|^2}{\|T_r \tilde u\|^2 } <\frac{1}{n}\sum_{j=0}^{n-1}  |\gamma_j(T_r)|^2
$$ 
This scenario is illustrated in Figure~\ref{fig:Alg2R2B} (left) with $r=1$. In this example, SNR $= 10$ and the target function is piecewise quadratic, making the ratio $\frac{\|T_2\tilde u\|^2}{\|T_1 \tilde u\|^2}$ smaller than 
$$
\frac{1}{ n}\sum_{j=0}^{n-1}  |\gamma_j(T_1)|^2 =  2.
$$
Notice that here we used the identity
$$
\frac{1}{n}\sum_{j=0}^{n-1}  |\gamma_j(T_r)|^2  = \frac{4^r }{\pi} \int_0^\pi \sin^{2r}(x)dx = 4^r  \frac{(2r-1)(2r-3)\dots 1}{(2r)(2r-2)\dots 2},
$$
which holds for $n>2r$ by exactness of the trapezoidal quadrature rule. Given the large SNR and the inadequate prior in this case, our algorithm favors forfeiting regularization.  This scenario is avoided when a better prior is specified. Figure~\ref{fig:Alg2R2B}  shows that, for the same data and $r=2$, $\lambda_{opt} \approx 2.47 $ is returned by the fixed point iteration. The corresponding denoised signal is presented in the right frame of Figure~\ref{fig:Alg2R2B}.

\begin{figure}[ht]
\centerline{\includegraphics[width=16cm]{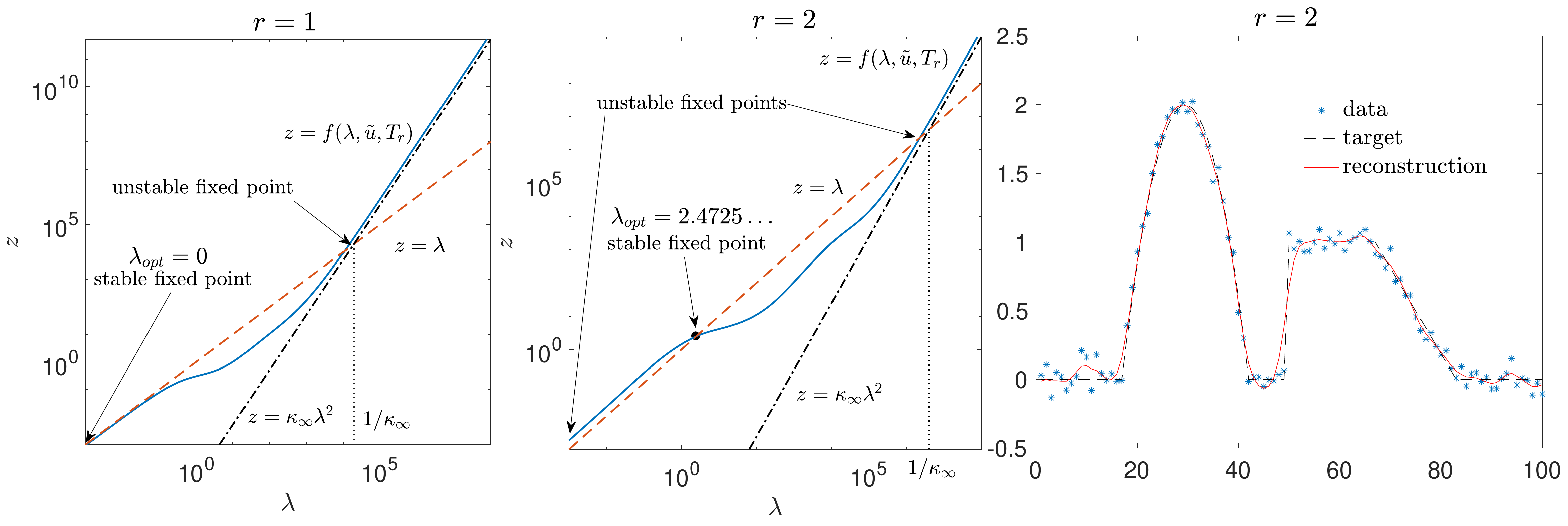}}
\caption{From left to right:  a loglog plot of the fixed point iteration function for $r=1$ showing 0 as the stable fixed point; a loglog plot of the fixed point  iteration function for $r=2$ showing a nonzero stable fixed point; and plots of the corresponding data $\tilde u$,  the target reconstruction, and the denoised signal using $r=2$ and $\lambda_{opt}$.}
\label{fig:Alg2R2B}
\end{figure}

\section{Conclusions}
This article considers the question of choosing
the weighting parameter for a regularization term in a least-squares problem.
One way to proceed is to take a Bayesian perspective and express the problem
as that of finding the maximum a posteriori (MAP) solution.
Proceeding in this way introduces, as a second parameter, the variance
of the Gaussian noise associated with the data of the least-squares term.
Then, one way of choosing these two parameters
is to maximize the probability of the data given the two parameters, which
here it is called the maximum evidence (ME) value.

We demonstrate the effectiveness of ME for image reconstruction problems in several instances
by comparison with UPRE.
Whereas,
the UPRE value assumes that the variance of the noise in the data is given,
the ME value
is obtained by determining the variance of the noise in the data automatically.
Simple test examples demonstrate the accuracy of the ME estimate for the variance of the noise.
The tests also show decidedly better reconstructions for ME than for UPRE
for certain types of data and regularizations, though we do not make any general claims concerning the relative merits of the two approaches.
The automatic choice of parameters is also tested for tomographic
imaging {\TLS{and deconvolution}} with good results.  

Also presented is an apparently novel iterative scheme,
and its efficient implementation, for determining the
ME values in the case of an $\ell_2$ regularization term.
Empirical evidence is presented demonstrating both rapid convergence
and insensitivity to an initial guess for the parameters.
The general algorithm requires repeated computationally expensive estimates of matrix traces.
At the same time, for many important inverse problems, including denoising, MRI, and deconvolution,
the trace and corresponding iterative estimates of the parameters can be calculated analytically with very little computation. In this case, the algorithm is extremely efficient and practical.  

Finally, it is shown how the Bayesian perspective facilitates
a very reasonable choice for the weighing parameter in the case
of $\ell_1$ regularization.
Experimental results
for both simple test examples and {\TLS{image deconvolution}}
give reconstructions of high quality
(and confirm the general superiority of $\ell_1$ regularization).  {\TLS{In the case of deconvolution, it is demonstrated that these $\ell_1$ parameters are \emph{nearly} optimal, and are determined at almost no notable computational expense to the $\ell_1$ solution.  This observation significantly broadens the practical application of the $\ell_2$ parameter selection work developed here.  Moreover, this approach may be trivially generalized to select parameters for the $\ell_1$ problem based on UPRE, whenever the noise level is known \emph{a priori}.}}

\begin{appendices}
\section{Extension of Algorithm 2 to 2D Image Denoising} \label{sec: 2Dext}
In the case of 2D images $u\in\R^{n^2}$, we consider regularizers that compute finite differences analogous to (\ref{1D-matrices}) along all rows and columns of the 2D image.  We write the order $k$ operator as 
$$
T_r^{2D} = \left[
\begin{array}{c}
T_{r,x} \\
T_{r,y}
\end{array}
\right],
$$
and we obtain the matrix $H = I + \lambda(T_{r,x}\T T_{r,x} + T_{r,y}\T T_{r,y})$.  The operators $T_{r,x}$ and $T_{r,y}$ are visualized in Figure \ref{fig: t-matrix} for the case $r=1$ and $n=4$.

\begin{figure}[ht]
\centering
 \includegraphics[width=.4\textwidth]{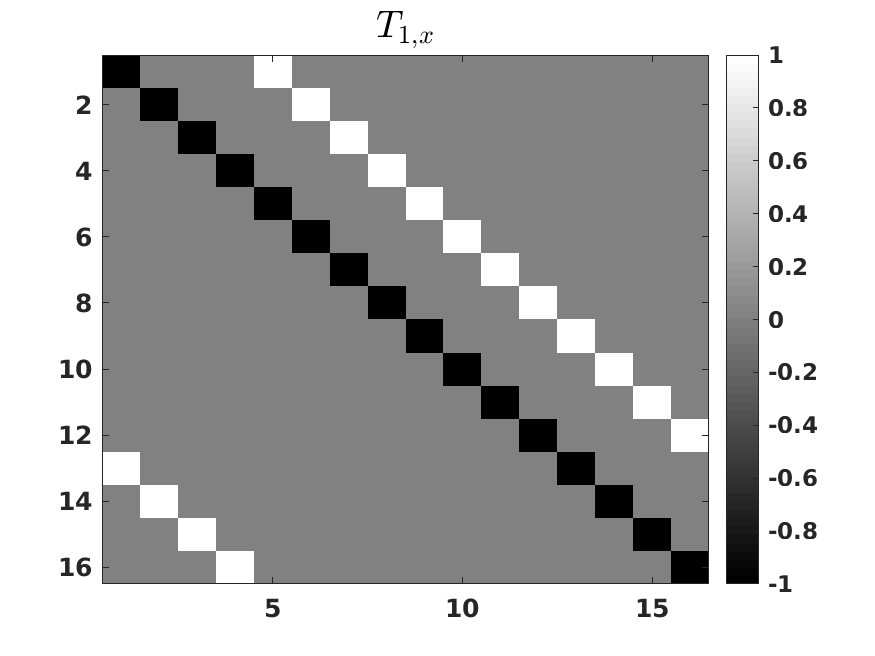}
 \includegraphics[width=.4\textwidth]{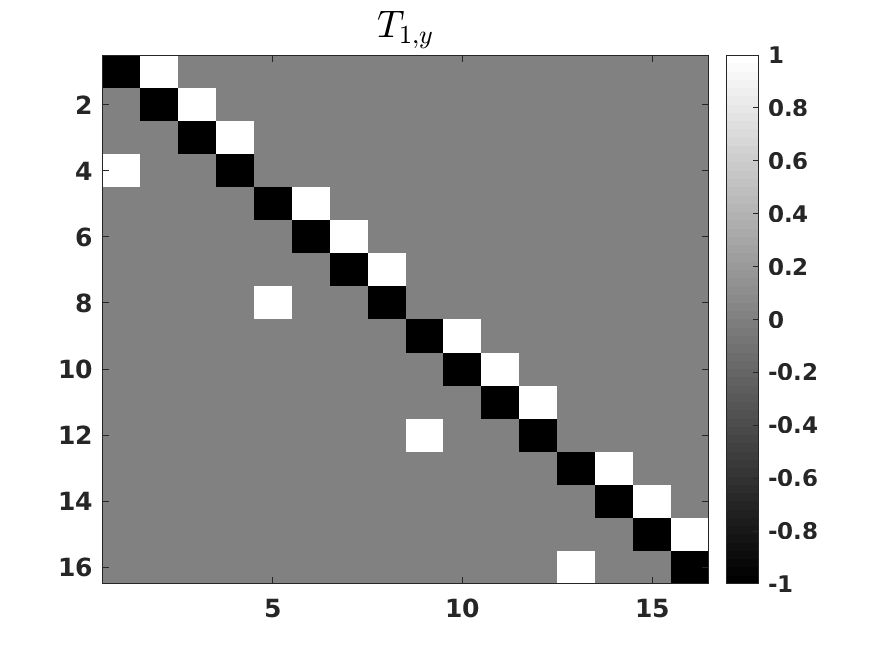}
 \caption{Visualization of the 2D $T_r$ operators for $r=1$ and $n=4.$}
 \label{fig: t-matrix}
\end{figure}

To determine the exact trace of $H^{-1}$ and the exact solution for (\ref{eq: tik}) using only FFT's, we will analytically determine the eigenvalues of $T_{r,x}\T T_{r,x} + T_{r,y}\T T_{r,y}$.  We first observe that these matrices have the following Kronecker product representation:
\begin{equation}\label{eq: kron}
T_{r,x} = T_r \otimes I, \quad T_{r,y} = I\otimes T_r,
\end{equation}
where $T_r$ is the 1D version of the difference operators as written in (\ref{1D-matrices}) and $I$ is the $n\times n$ identity.  This leads to the product representations as $T_{r,x}\T T_{r,x} = T_r\T T_r \otimes I$ and $T_{r,y}\T T_{r,y} = I\otimes T_r\T T_r $.  To find the solution to (\ref{eq: tik}) in this case, combine the Kronecker representation with the unitary Fourier diagonalization of $T_r$ to obtain
\begin{align}
 T_{r,x}\T T_{r,x} & = (\F^{-1} \otimes \F^{-1} )(|\Lambda_r|^2 \otimes I)(\F \otimes \F) ,\\
  T_{r,y}\T T_{r,y} & = (\F^{-1} \otimes \F^{-1} )(I \otimes |\Lambda_r|^2)(\F \otimes \F).
\end{align}
This leads to our expression for $u_\lambda$ in (\ref{eq: tik}) in the case of 2D denoising as
\begin{equation}\label{eq: tiksol2D}
 u_\lambda = (\F^{-1} \otimes \F^{-1} ) \left[ I + \lambda (|\Lambda_r|^2 \otimes I + I\otimes |\Lambda_r|^2)\right]^{-1}(\F \otimes \F) \tilde u 
\end{equation}

Observe that $\F \otimes \F$ and $\F^{-1} \otimes \F^{-1}$ are the 2D unitary discrete Fourier and inverse Fourier transforms, therefore analogous to the 1D case this solution requires two 2D FFTs and product with a diagonal matrix, whose values we determine in what follows.

Using properties of Kronecker products, it can be easily shown that for any two eigenvalues $\gamma_1, \gamma_2$ of $T_r\T T_r$, then $\gamma_1 + \gamma_2$ is an eigenvalue of $T_{r,x}\T T_{r,x} + T_{r,y}\T T_{r,y}$, and therefore the complete set of eigenvalues of this matrix are obtained by considering all such combinations.
Combining this observation with the eigenvalues in (\ref{eq: eig1d}) leads us to the exact trace as
\begin{equation}\label{eq: trace2D}
\text{trace}(H^{-1}) = \sum_{j=0}^{n-1} \sum_{k =0}^{n-1} \left[1 + 4^r\lambda\left( \sin^{2r}(\pi j/n) + \sin^{2r}(\pi k /n) \right) \right]^{-1} ,
\end{equation}
and similarly the trace of $H^{-1} T\T T $ is given by
\begin{equation}\label{eq: trace2D2}
\text{trace}(H^{-1}T\T T) =  \sum_{j=0}^{n-1} \sum_{k=0}^{n-1} \left[4^{-r}+ \lambda\left( \sin^{2r}(\pi j/n) + \sin^{2r}(\pi k /n) \right) \right]^{-1}  \left( \sin^{2r}(\pi j/n) + \sin^{2r}(\pi k/n) \right).
\end{equation}
Finally, the entries of the diagonal matrix needed in evaluation of (\ref{eq: tiksol2D}) coincide with the terms in the sum (\ref{eq: trace2D}).  The arguments used here are very easily extended to higher dimensions, say 3D video denoising.  In addition, they also extend to other circulant regularization matrices, such as the more effective multiscale operators in \cite{sanders2017multiscale}.

\section{Extension of Algorithm 2 to Fourier Sampling}
Consider our reconstruction problem in the case that the sampling matrix is $A = P\F$, where $P$ is a row selector matrix, i.e. the identity with some rows deleted.  In other words, we have some Fourier coefficients of $u$, and let $S\subset \{1,2,\dots, n\}$ denote the indices of those rows of the identity that are in $P$.  Then using some similar arguments as before, the solution to (\ref{eq: tik}) is given by
\begin{equation}\label{eq: tik-Fourier}
 u_\lambda = \F^{-1} (P\T P + \lambda |\Lambda_r|^2)^{-1} \F b.
\end{equation}
Hence, again in this case one only needs two FFTs and a product with a diagonal matrix to obtain the exact solution.  These entries are easily seen once again using (\ref{eq: eig1d}).  Moreover, the traces needed are given by
\begin{equation}
 \text{trace}(H^{-1} A^* A) = \sum_{j\in S} \left( 1+ \lambda 4^r \sin^{2r} (\pi (j-1)/n) \right)^{-1} ,
\end{equation}

\begin{equation}
 \text{trace}(H^{-1} T_r\T T_r ) = \sum_{j=0}^{n-1} \left( 4^{-r} \delta_S (j+1) + \lambda \sin^{2r} (\pi j/n) \right)^{-1}  \sin^{2r} (\pi j/n),
\end{equation}
where $\delta_S (j)=1$ for $j\in S$ and 0 otherwise.  These concepts are extended to 2D and higher dimensions repeating similar arguments from section \ref{sec: 2Dext}.

\section{Proof of Convergence Results}

\begin{proof}[proof of Proposition \ref{prop:alg2}] 
The equations of Algorithm~2, together with $T =\F^{-1}\Lambda\F$, gives
\begin{eqnarray*}
\lambda_{k+1}&=&\frac{\|u^\ast_k - \tilde{u}\|_2^2}{\tr(I - H_k^{-1})}
 \frac{\tr(I - \lambda_k H_k^{-1}T\T T)}{\|T u^\ast_k\|_2^2} \\
 &=&\frac{\|-\lambda_k T\T T H_k^{-1}\tilde{u}\|_2^2}%
{\|T H_k^{-1}\tilde{u}\|_2^2}
 \frac{\tr(H_k^{-1})}{\tr(\lambda_k T\T T H_k^{-1})} \\
 &=&\lambda_k\frac{\||\Lambda|^2 B(\lambda_k)\hat{u}\|_2^2}%
{\||\Lambda| B(\lambda_k)\hat{u}\|_2^2}
 \frac{\tr(B(\lambda_k))}{\tr(|\Lambda|^2 B(\lambda_k))}
\end{eqnarray*}
where $B(\lambda) = (I +\lambda|\Lambda|^2)^{-1}$.
\end{proof}

\begin{proof}[proof of Proposition \ref{prop:PropertiesAlg2}] The first part of the proposition follows directly by substituting $|\Lambda|^2 = I$ into \eqref{eq:alg2A}. We now focus on the two other parts. Again from \eqref{eq:alg2A} it is immediately deduced that 0 is a fixed point of $f$. The stability of this fixed point depends on the derivative of $f$, which is given by

\[\lim_{\lambda\rightarrow 0}\frac{f(\lambda)}{\lambda}
 =\frac{n\||\Lambda|^2\hat{u}\|_2^2}%
{\tr(|\Lambda|^2)\||\Lambda|\hat{u}\|_2^2}
 = n\frac{\|T\T T\tilde{u}\|_2^2}{\tr(|\Lambda|^2)\|T\tilde{u}\|_2^2}.\]

Finally, proving part 3 makes repeated use of the facts $\gamma_0 = 0$ for $r>0$ and $\lim_{\lambda\rightarrow \infty} \lambda B(\lambda) = |\Lambda_r|^{-2}$, where here we use the minimium norm pseudo inverse. To this end, we evaluate the limit:
\begin{align*}
 \lim_{\lambda\rightarrow\infty}\frac{f(\lambda)}{\lambda^2}
&  =\lim_{\lambda\rightarrow\infty}
 \frac{\||\Lambda|^2\lambda B(\lambda)\hat{u}\|_2^2}%
{\||\Lambda|\lambda B(\lambda)\hat{u}\|_2^2}
\frac{\tr(B(\lambda))}{\tr(|\Lambda|^2\lambda B(\lambda))} \\
& =  \frac{ \sum_{j=1}^{n-1} |\hat{u}_j|^2 }{ \sum_{j=1}^{n-1} |\hat{u}_j|^2/ |\gamma_j|^2}
\frac{1}{ n-1 }
= \kappa_{\infty}(\tilde u,T_r),
\end{align*}
It follows that
$$
\frac{\min_{i=1,\dots, n-1}  |\gamma_i|^2 \sum_{j=1}^{n-1} |\hat{u}_j|^2 }{ (n-1)  \sum_{j=1}^{n-1} |\hat{u}_j|^2} \le \kappa_{\infty}(\tilde u,T_r) \le \frac{\max_{i=1,\dots, n-1}  |\gamma_i|^2 \sum_{j=1}^{n-1} |\hat{u}_j|^2 }{ (n-1)  \sum_{j=1}^{n-1} |\hat{u}_j|^2} \\
$$
or
$$
 \frac{4^r\sin^{2r}(\pi/n)}{n-1} \le \kappa_{\infty}(\tilde u,T_r)  \le \frac {4^r}{n-1}.
 $$
\end{proof}

\end{appendices}

\section*{Acknowledgements}
This work is supported in part by the grants NSF-DMS 1502640 and AFOSR FA9550-15-1-0152.


\end{document}